\documentclass[12pt,draft]{article}
\usepackage{tikz-cd}
\usepackage{amsmath}
\usepackage{amssymb}
\usepackage{amsthm}
\usepackage{geometry}
\usepackage{titlesec}
\usepackage[numbers,sort&compress]{natbib}
\usepackage[all]{xy}
\usepackage{amsfonts}
\usepackage[numbers,sort&compress]{natbib}
\usepackage{cases}
\usepackage{amsthm}
\usepackage[active]{srcltx}
\usepackage{autobreak}
\usepackage[shortlabels]{enumitem}
\numberwithin{equation}{section}
\theoremstyle{definition}

\titleformat*{\section}{\large\bfseries}
\titleformat*{\subsection}{\normalsize\bfseries}

\def\la{\lambda}

\def\Z{\mathbb{Z}}
\def\N{\mathbb{N}}

\def\L{\mathcal{L}}
\def\G{\mathcal{G}}
\def\Glm{\mathcal{G}_{\lambda,\mu}}

\def\C{\mathbb{C}}

\def\mi{\mathbf{i}}
\def\mj{\mathbf{j}}
\def\mk{\mathbf{k}}
\def\mx{\mathbf{x}}
\def\my{\mathbf{y}}
\def\mz{\mathbf{z}}
\def\Sj{\hat{s}_{j}}
\def\Re{\mathrm{Re}}

\numberwithin{equation}{section}
\newtheorem{theo}{Theorem}[section]
\newtheorem{defi}[theo]{Definition}
\newtheorem{coro}[theo]{Corollary}
\newtheorem{lemm}[theo]{Lemma}
\newtheorem{prop}[theo]{Proposition}
\newtheorem{clai}{Claim}

\newtheorem{rema}[theo]{Remark}

\newtheorem{remark}[theo]{Remark}

\begin{document}
	
	\begin{center}
		\bfseries\large
		{Simple restricted modules   over the deformative Schr\"{o}dinger-Virasoro algebra\,{$^1\,$}}
		\footnote{$^1\,$Supported by the National Natural Science Foundation of China (No. 12471027) and Natural Science Foundation of Shanghai (No. 24ZR1471900).\\\indent \ \  $\ast\, $Corresponding author: Xiaoqing Yue (xiaoqingyue@tongji.edu.cn)}
		
		\mdseries\normalsize
		\bigskip
		Haibo Chen$^{\dagger}$, Yongtao Liu$^{\ddagger,\S}$, Xiaoqing Yue$^{\ddagger,\S,\ast}$
		
		\footnotesize
		\smallskip
		$\dagger$ School of Science, Jimei University, Xiamen, Fujian 361021, China
		
		\smallskip
		$\ddagger$ School of Mathematical Sciences, Tongji University, Shanghai 200092, China
		
		\smallskip
		$\S$ Key Laboratory of Intelligent Computing and  Applications(Tongji University), Ministry of Education
		
		\smallskip
		E-mails:hypo1025@jmu.edu.cn, 2211178@tongji.edu.cn, xiaoqingyue@tongji.edu.cn
	\end{center}
	{
		\footnotesize
		\noindent\textbf{Abstract}. This paper investigates simple restricted modules over the deformed Schr\"{o}dinger-Virasoro algebra $\Glm$, which gives a complete classification of them  for some  $\lambda,\mu\in\mathbb{C}$. More precisely,  we provide a systematic construction of these modules, including highest weight modules and Whittaker modules, by inducing simple modules from the positive part's quotient algebras. We prove that any simple restricted $\Glm$-module satisfying certain injective conditions is isomorphic to such an induced module.   As an application,  we obtain some simple weak $V(c)$-modules over vertex  algebras associated to $\mathcal{G}_{\lambda,\mu}$ for some $\lambda,\mu\in\mathbb{C}$.   Note that our results include the Schr\"{o}dinger-Virasoro algebra and the deformed $\mathfrak{bms}_3$ algebra as special cases, thereby improving upon some of the previously reported results of \cite[Theorem 3.4]{CHS} and \cite[Theorem 2]{Cqf}. This work effectively classifies and generalizes the representation theory of the deformed family.
		
		
		
		%	17B05  	Structure theory for Lie algebras and superalgebras
		%	17B40  	Automorphisms, derivations, other operators for Lie algebras and super algebras
		%	17B65  	Infinite-dimensional Lie (super)algebras [See also 22E65]
		%	17B68  	Virasoro and related algebras
	}
	
	\section{Introduction}\label{sec1}
	The Schr\"{o}dinger-Virasoro algebra is an infinite-dimensional Lie algebra that plays a fundamental role in physics and conformal field theory \cite{ LS, LS2, RU, TZ, Uj, ZT}. It was originally introduced by Henkel in the study of free Schr\"{o}dinger equations and describes the local invariance of certain physical systems \cite{Hm}. The Schr\"{o}dinger-Virasoro algebra is an extension of the Virasoro algebra by a nilpotent Lie algebra consisting of a bosonic current of weight $3/2$ and a bosonic current of weight $1$.  Since then, its representation theory, including weight modules, Verma modules, and Whittaker modules, has been extensively investigated by numerous authors \cite{CGLW, ALZ, CHS, LS, RU, TZ, ZT, ZTL}. In 2006, Roger and Unterberger introduced a three-parameter deformation of this algebra, denoted as $\Glm$, which arises naturally from the deformation of the contact structure on the $(1+1)$-dimensional space-time \cite{RU}. The study of its representation theory is not only a natural mathematical extension but also essential for understanding how the deformation parameters $\lambda$ and $\mu$ affect the underlying algebraic properties. 
	
	One of the most important categories in the representation theory of infinite-dimensional Lie algebras is the category of smooth modules (also referred to as restricted modules) (see, e.g., \cite{CG, FGXZ, MZ, LPX, CYZ, MNTZ}). This category is not only a natural generalization of highest weight modules but also serves as the bridge between Lie algebras and vertex operator algebras (VOAs) \cite{Kac, LL}. The characterization of simple restricted modules has been studied for a variety of algebraic structures, such as the Virasoro algebra \cite{MZ}, the Heisenberg-Virasoro algebra \cite{CG,TYZ}, the Neveu-Schwarz algebra \cite{LPX2}, and the planar Galilean conformal algebra \cite{GG,CY}. Recent developments also include the study of restricted modules for other deformed structures such as the deformed $\mathfrak{bms}_3$ algebra \cite{Cqf}.
	
	Now we recall the definition of the {\em deformative Schr\"{o}dinger-Virasoro algebra} $\Glm(\epsilon) \,(\lambda,\mu$ $\in\C,\epsilon=\{0,\frac{1}{2}\})$ \cite{RU}, which is
	an infinite-dimensional  Lie algebra
	with the $\C$-basis  $\{M_m,Y_{m-\epsilon},$ $L_m,C
	\mid m\in \Z\}$ and the following   Lie brackets:
	\begin{equation}\label{def1.1}
		\aligned
		&[L_m,L_n]= (n-m)L_{m+n}+\delta_{m+n,0}\frac{m^{3}-m}{12}C,\\&
		[L_m,Y_{n-\epsilon}]= \Big(n-\frac{\lambda+1}{2}m+\mu-\epsilon\Big)Y_{m+n-\epsilon},
		\\&
		[Y_{m-\epsilon},Y_{n-\epsilon}]= (m-n)M_{m+n-2\epsilon},\\&
		[L_m,M_n]=(n-\lambda m+2\mu) M_{m+n},
		\\&
		[M_m,M_n]=[M_m,Y_{n-\epsilon}]=[\Glm(\epsilon),C]=0, \quad \forall\, m,n\in\Z.
		\endaligned
	\end{equation}
	Note that  the center of $\Glm(\epsilon)$ is spanned by $C$. Denote $\Glm^{(r)}(\epsilon)=\oplus_{i\in\Z}(\delta_{i,r}\C L_{i}\oplus \delta_{i+\epsilon,r}\C Y_{i+\epsilon} \oplus \delta_{i,r}\C M_{i}) \oplus \delta_{0,r}\C C$ for all $r\in\Z+\epsilon$, then we have $\Glm(\epsilon)=\oplus_{r\in\Z+\epsilon}\,\Glm^{(r)}(\epsilon)$. Clearly, the subalgebra of $\mathcal{G}_{\lambda,\mu}(\epsilon)$ spanned by $\{L_m,M_m,C\mid m\in\mathbb{Z}\}$  is isomorphic to the Lie algebra  $W(a,b)$, which includes a lot of important Lie algebras such as Virasoro algebra, $W$-algebra $W(2,2)$, twisted Heisenberg-Virasoro algebra and so on.   In addition, the  Schr\"{o}dinger-Virasoro algebra $\G_{0,0}(\frac{1}{2})$ and  the deformed $\mathfrak{bms}_3$ algebra $\G_{1,0}(0)$ are special cases for the
	deformative Schr\"{o}dinger-Virasoro algebra (see \cite{TZ,Cqf}).
	
	From  $\Glm(0)\cong\Glm(\frac{1}{2})$, we only consider  Lie algebra $\Glm:=\Glm(\frac{1}{2})$  throughout  this paper. In \cite{LS},  the authors explored the deep connection between modules over $\Glm$ and vertex  algebras. They  employed Li's theory of local systems \cite{Lhs} to associate $\Glm$ with a vertex algebra $V(h, c)$, which showed that there is a one-to-one correspondence between restricted $\Glm$-modules and related vertex algebra modules. In the present paper,  we investigate the simple restricted modules over the deformed Schr\"{o}dinger-Virasoro algebra $\Glm$, and explore their applications in the theory of vertex algebras.   Unlike the standard Schr\"{o}dinger-Virasoro case, the deformation parameters $\lambda$ and $\mu$ introduce non-trivial coefficients into the Lie brackets $[L_m, M_n]$ and $[L_m, Y_{p}]$. Besides, we show that any simple restricted module satisfying these conditions is necessarily isomorphic to one of the modules in our construction, thereby providing an explicit characterization of this category.

	The paper is organized as follows. In Section \ref{Sec2}, we provide necessary preliminaries. In Section \ref{Sec3}, we construct a family of simple restricted modules and provide the technical lemmas required for proving simplicity through degree analysis in Theorem \ref{th1}. We   give a  complete classification of simple restricted modules  over $\Glm$ for some $\lambda,\mu\in\mathbb{C}$ in Theorems \ref{th2} and \ref{th3}. In Section \ref{Sec4}, from the equivalence between the category of restricted modules and vertex algebra modules, we give the structure of a weak $V(c)$-module from a restricted $\Glm$-module. Finally, in Section \ref{Sec5}, we discuss some specific examples, including highest weight modules, Whittaker modules and some certain cases of $(\la,\mu)$. It is worth noting that when $\mu\neq0$, the highest weight module in the present context exhibits a non-standard construction form.

	Throughout this paper, we denote by $\C,\Z,\N$  and $\Z_+$ the sets of complex numbers, integers, nonnegative integers and positive integers, respectively. We use $(a,b)=(c,d)$ to represent $a=c$ and $b=d$ for convenience. For a complex number $z$, we denote by $\Re(z)$  the real part of $z$.
	
	\section{Preliminaries}\label{Sec2}
	In this section, we shall construct a class of   induced   $\Glm$-modules $\mathrm{Ind}(V)$, where $V$ is a simple module over a certain subalgebra of $\mathcal{G}_{\lambda,\mu}$. First, we recall some definitions and  results  for later use.
	\begin{defi}\rm
		Let $V$ be a module of a Lie algebra $\L$ and $x\in\L$.
		
		{\rm (1)} If for any $v\in V$ there exists $n\in\Z_+$ such that $x^nv=0$,  then we call that the action of  $x$  on $V$ is {\em locally nilpotent}.
		Similarly, the action of $\L$   on $V$  is {\em locally nilpotent} if for any $v\in V$ there exists $n\in\Z_+$ such that $\L^nv=0$.
		
		{\rm (2)} If for any $v\in V$ we have $\mathrm{dim}(\sum_{n\in\Z_+}\C x^nv)<+\infty$, then we call that the action of $x$ on $V$ is  {\em locally finite}.
		Similarly, the action of $\L$  on $V$  is  {\em locally finite}
		if for any $v\in V$ we have $\mathrm{dim}(\sum_{n\in\Z_+} {\L}^nv)<+\infty$.
	\end{defi}
	One can check that  the action of $x$ on $V$ is locally nilpotent implies that the action of $x$ on $V$ is locally finite.
	If $\L$ is a finitely generated Lie algebra, then we have that the action of $\L$ on $V$ is locally nilpotent implies that the action of
	$\L$
	on $V$ is locally finite.
	Now we  recall the definition of restricted modules.
	\begin{defi}\rm
		Let $\L=\oplus_{i\in\frac{1}{2}\Z}\L_i$ be a $\frac{1}{2}\Z$-graded Lie algebra. An $\L$-module $V$ is called the {\em restricted module} if for any $v\in V$,  there exists $n\in\N$ such that $\L_{n+i}v=0$ for all $i\in\frac{1}{2}\Z_+$.
	\end{defi}
	
	\begin{defi}\rm
		Let $\L$ be a Lie algebra with triangular decomposition $\L=\L^{+}\oplus\L_{0}\oplus\L^{-}$. Let $V$ be an $\L$-module and let $\psi:\L^{+}\rightarrow \C$ be a Lie algebra homomorphism. A vector $v\in V$ is called a \emph{Whittaker vector} if $xv=\psi(x)v$ for every $x\in\L^{+}$. An $\L$-module $V$ is called a \emph{Whittaker module of type} $\psi$ if there is a Whittaker vector $w\in V$ which generates $V$. In this case, we call $w$ a \emph{cyclic Whittaker vector}.
	\end{defi}
	
	Denote by  $\mathbf{S}$  the set of all infinite vectors of the form $\mi:=(\ldots, i_2, i_1)$ with entries in $\N$,
	satisfying the condition that the number of nonzero entries is finite. Let $\mathbf{0}$ be the element $(\ldots, 0, 0)\in\mathbf{S}$. For
	$i\in\Z_+$, denote $\epsilon_i=(\ldots,0,1,0,\ldots,0)\in\mathbf{S}$,
	where $1$ is
	in the $i$'th  position from the right. For any $\mi\in\mathbf{S}$, we write
	$$\mathbf{w}(\mi)=\sum_{s\in\Z_+}s\cdot i_s,\ \ \ \mathbf{d}(\mi)=\sum_{s\in\Z_+} i_s,$$
	which are nonnegative integers. For any nonzero  $\mi\in\mathbf{S}$, let $p$ and $q$ be the largest and smallest integers such that $i_p	 \neq0$ and $i_q	\neq0$ respectively,
	and define  $\mi^\prime=\mi-\epsilon_p$
	and  $\mi^{\prime\prime}=\mi-\epsilon_q$.
	
	The following definitions of some total order are given in \cite{MZ,CG}. We always consider that  $\mathbf{0}$ is the minimum element.
	\begin{defi}\rm\label{def2.3}
		{\rm (1)} Denote by $>$ the   {\em lexicographical total order}  on  $\mathbf{S}$, defined as follows: for any $\mi,\mj\in\mathbf{S}$
		$$\mj >  \mi \ \Leftrightarrow \ \mathrm{ there\ exists} \ r\in\Z_+ \ \mathrm{such \ that} \ (j_s=i_s,\ \forall\, s>r) \ \mathrm{and} \ j_r>i_r.$$
		{\rm (2)} Denote by $\succ$ the   {\em reverse  lexicographical  total order}  on  $\mathbf{S}$,  defined as follows: for any $\mi,\mj\in\mathbf{S}$
		$$\mj \succ \mi \ \Leftrightarrow \ \mathrm{ there\ exists} \ r\in\Z_+ \ \mathrm{such \ that} \ (j_s=i_s,\ \forall\, 1\leq s<r) \ \mathrm{and} \ j_r>i_r.$$
		Now we can induce a {\em principal total order} on $\mathbf{S}\times\mathbf{S}\times\mathbf{S}$, still denoted by $\succ$:
		\begin{equation*}\aligned
			(\mi,\mj,\mk) \succ (\mathbf{l},\mathbf{m},\mathbf{n})\  \Leftrightarrow \
			\ &(\mk,\mathbf{w}(\mk))\succ(\mathbf{n},\mathbf{w}(\mathbf{n})) \quad \mathrm{or}\\&
			\mk=\mathbf{n}\
			\mathrm{and}\ (\mj,\mathbf{w}(\mj)) \succ (\mathbf{m},\mathbf{w}(\mathbf{m}))
			\quad \mathrm{or}\\&
			\mk=\mathbf{n}, \ \mj=\mathbf{m} \
			\mathrm{and }\ \mi > \mathbf{l},\quad \forall\,\mi,\mj,\mk,\mathbf{l},\mathbf{m},\mathbf{n}\in\mathbf{S}
			%NEW
			%(i_2,i_1,\mj,\mk) \succ (l_2,l_1,\mathbf{m},\mathbf{n})\  \Leftrightarrow\
			%\ &(\mk,\mathbf{w}(\mk))\succ(\mathbf{n},\mathbf{w}(\mathbf{n})) \quad \mathrm{or}\\&
			%\mk=\mathbf{n}\
			%\mathrm{and}\ (\mj,\mathbf{w}(\mj)) \succ (\mathbf{m},\mathbf{w}(\mathbf{m}))
			%\quad \mathrm{or}\\&
			%\mk=\mathbf{n}, \ \mj=\mathbf{m} \
			%\mathrm{and }\ i_1 > l_1,
			%\quad \mathrm{or}\\&
			%\mk=\mathbf{n}, \ \mj=\mathbf{m} \
			%\mathrm{and }\ i_1 = l_1\ \mathrm{and }\ i_2 > l_2 ,\\&
			%\quad
			%\ \mathrm{and}\ i_1,i_2,l_1,l_2\in\N
			.
			\endaligned
		\end{equation*}
	\end{defi}
	For any $d_1,d_2\in\N$ with $d_1\geq2d_2-1$,   set
	$$\Glm^{(d_1,d_2)}=\sum_{i\in\N}(\C M_{i-d_1}\oplus\C (1-\delta_{i,0})Y_{i-d_2-\frac{1}{2}}\oplus\C L_i)\oplus\C C.$$
	It is clear that $\Glm^{(d_1,d_2)}$ is a subalgebra of $\Glm$.
	
	Letting  $V$ be a simple $\Glm^{(d_1,d_2)}$-module,  we have the following  induced $\Glm$-module
	$$\mathrm{Ind}(V)=\mathcal{U}(\Glm)\otimes_{\mathcal{U}(\Glm^{(d_1,d_2)})}V.$$
	
	Since simple modules over one of the subalgebras of $\Glm$ containing the central element $C$ are usually considered in the sequel,
	we always assume that the action of $C$ is scalar $c$, and we say the modules are of level $c$.

	Fix $d_1,d_2\in\N$  and let $V$ be a simple $\Glm^{(d_1,d_2)}$-module. For $\mi,\mj,\mk\in\mathbf{S}$, we denote
	$$M^{\mi} Y^{\mj} L^{\mk}=\ldots M_{-d_1-2}^{i_2} M_{-d_1-1}^{i_1}\ldots Y_{-d_2-\frac{3}{2}}^{j_2} Y_{-d_2-\frac{1}{2}}^{j_1}\ldots L_{-2}^{k_2} L_{-1}^{k_1}\in \mathcal{U}(\Glm).$$
	%$$M^{(i_2,i_1)} Y^{\mj} L^{\mk}=M_{-d_1-2}^{i_2} M_{-d_1-1}^{i_1}\ldots Y_{-d_2-\frac{3}{2}}^{j_2} Y_{-d_2-\frac{1}{2}}^{j_1}\ldots L_{-2}^{k_2} L_{-1}^{k_1}\in \mathcal{U}(\Glm).$$
	%For $M_{-d_1-t-2}$ and $t\in\Z_+$, since $d_1\geq2d_2-1$, by (\ref{def1.1}) we have
	%\begin{equation*}\aligned
	%	M_{-d_1-t-2}V&=\frac{1}{d_1+t+1-2d_2}[Y_{-d_1+d_2-t-\frac{3}{2}},Y_{-d_2-\frac{1}{2}}]V\\&
	%	=\frac{1}{d_1+t+1-2d_2}(Y_{-d_1+d_2-t-\frac{3}{2}}Y_{-d_2-\frac{1}{2}}-Y_{-d_2-\frac{1}{2}}Y_{-d_1+d_2-t-\frac{3}{2}})V.
	%	\endaligned
	%\end{equation*}
	It follows from the $\mathrm{PBW}$ Theorem that every element of $\mathrm{Ind}(V)$ can be uniquely written as  the
	form
	\begin{equation}\label{def2.1}
		\sum_{\mi,\mj,\mk\in\mathbf{S}}M^{\mi} Y^{\mj} L^{\mk} v_{\mi,\mj,\mk},
		%\sum_{i_2,i_1\in\N,\ \mj,\mk\in\mathbf{S}}M^{(i_2,i_1)} Y^{\mj} L^{\mk} v_{i_2,i_1,\mj,\mk},
	\end{equation}
	where all  $v_{\mi,\mj,\mk}\in V$ and only finitely many of them are nonzero. For any $v\in\mathrm{Ind}(V)$ as in  \eqref{def2.1}, we denote by $\mathrm{supp}(v)$ the set of all $(\mi,\mj,\mk)\in \mathbf{S}\times \mathbf{S}\times\mathbf{S}$  such that $v_{\mi,\mj,\mk}\neq0$.
	For a nonzero $v\in \mathrm{Ind}(V)$, we write $\mathrm{deg}(v)$  the maximal (with respect to the principal total order on $\mathbf{S}\times\mathbf{S}\times\mathbf{S}$) element in $\mathrm{supp}(v)$, called the {\em degree} of $v$. Note that here and later
	we make the
	convention that  $\mathrm{deg}(v)$ is defined only for  $v\neq0$.
	
	\begin{lemm}\label{le2.4}
		$\Glm\cong \G_{\lambda,\mu+m}$, $\forall\, m\in\Z$.
	\end{lemm}
	
	\begin{proof}
		For any $m\in\Z$, let $\varphi$ be a linear map from $\Glm$ to $\G_{\lambda,\mu+m}$ defined by
		$$\varphi\ :\ L_n\mapsto L_n,\ Y_{n-\frac{1}{2}}  \mapsto Y_{n-m-\frac{1}{2}},\ M_{n}  \mapsto M_{n-2m},\ C\mapsto C,$$
		where $n\in\Z$.
		It is straightforward to verify that $\varphi$ is a Lie algebra isomorphism.
	\end{proof}
	\begin{remark}\label{rema1}
		One can check if taking $m=-\frac{1}{2}$, by the same method, we have $\Glm(\frac{1}{2})\cong \G_{\la,\mu-\frac{1}{2}}(0)$. As mentioned earlier, we only need to discuss $\Glm(\frac{1}{2})$. Besides, without loss of generality, we can assume $\Re(\mu)\in[-1,0)$ from now on, which is useful for the proof of Theorem \ref{th2}.
	\end{remark}

	\section{Characterization of simple restricted modules}\label{Sec3}
	The purpose of this section is to present two main results of this paper. We first prove that the   induced $\Glm$-module $\mathrm{Ind}(V)$ is simple  under certain   conditions which  are stated in  Theorem \ref{th1}.  Then we shall show that under the conditions in  Theorem \ref{th1}, any simple $\Glm$-module with locally finite actions of elements $M_i,(1-\delta_{j,0})Y_{j-\frac{1}{2}},L_k$ for sufficiently large $i,j,k\in\N$ is isomorphic to one of the induced  $\Glm$-modules.
	
	We first prove the following lemma.
	\begin{lemm}\label{LYM}
		Let $V$ be a $\Glm$-module. For any $m,n\in\Z$ and $i\in\Z_+$, we have the following relations in $\mathcal{U}(\Glm)$:
		\begin{equation}\label{YYLM}
			\aligned
			&[Y_{m-\frac{1}{2}},Y_{n-\frac{1}{2}}^{i}]=i(m-n) Y_{n-\frac{1}{2}}^{i-1}M_{m+n-1},\\&
			[L_m,M_n^i]=i(n-\lambda m+2 \mu)M_n^{i-1}M_{m+n},
			\endaligned
		\end{equation}
		and
		\begin{align}
			&[M_m,L_n^i]=\sum_{s=1}^{i}\binom{i}{s}\prod_{t=0}^{s-1}\big( (\lambda-t) n-m-2\mu \big)L_n^{i-s} M_{m+sn}\label{MLYL1},\\&
			[Y_{m-\frac{1}{2}},L_n^i]=\sum_{s=1}^{i}\binom{i}{s}\prod_{t=0}^{s-1}\Big(\frac{\lambda+1-2t}{2}n+\frac{1}{2}-m-\mu\Big)L_n^{i-s} Y_{m+sn-\frac{1}{2}}\label{MLYL2}.
		\end{align}
	\end{lemm}
	\begin{proof}
		We prove this lemma by induction on $i\in\Z_+$. For $i=1$, (\ref{YYLM})--(\ref{MLYL2}) hold trivially. Assume the statement is true for some $k \in \Z_+$, i.e., (\ref{YYLM})--(\ref{MLYL2}) hold when $i=k$. Then for $i=k+1\geq 2$, we have the Leibniz rule $[x, yz] = [x,y]z + y[x,z]$ for all $x, y, z \in \mathcal{U}(\Glm)$.  By $Y_{n-\frac{1}{2}}M_{m+n-1}=M_{m+n-1}Y_{n-\frac{1}{2}}$ and assumption, we can deduce that
		\begin{align*}
			[Y_{m-\frac{1}{2}},Y_{n-\frac{1}{2}}^{k+1}]=\, &[Y_{m-\frac{1}{2}},Y_{n-\frac{1}{2}}^{k}]Y_{n-\frac{1}{2}}+Y_{n-\frac{1}{2}}^{k}[Y_{m-\frac{1}{2}},Y_{n-\frac{1}{2}}]\\
			=\,&k(m-n) Y_{n-\frac{1}{2}}^{k-1}M_{m+n-1}Y_{n-\frac{1}{2}}+(m-n)Y_{n-\frac{1}{2}}^{k}M_{m+n-1}\\
			=\,&(k+1)(m-n) Y_{n-\frac{1}{2}}^{k}M_{m+n-1}.
		\end{align*}
		Similarly, one can get $[L_m,M_n^i]=i(n-\lambda m+2 \mu)M_n^{i-1}M_{m+n}$.
		
		For formula (\ref{MLYL1}), we can get
		\begin{align}
			[M_{m},L_{n}^{k+1}]=\,&[M_{m},L_{n}^{k}]L_{n}+L_{n}^{k}[M_{m},L_{n}]\nonumber\\
			=\,&[M_{m},L_{n}^{k}]L_{n}-(m-\la n+2\mu)L_{n}^{k}M_{m+n}.\label{equ3.4}
		\end{align}
		For convenience, let $A_{s} = \prod_{t=0}^{s-1}\big((\lambda-t)n-m-2\mu\big)$. Note that $A_{s} \cdot \big((\lambda-s)n - m - 2\mu\big) = A_{s+1}$. Then by induction hypothesis, (\ref{equ3.4}) turns into:
		\begin{align*}
			[M_m, L_n^{k+1}] = & \sum_{s=1}^{k} \binom{k}{s} A_{s} L_n^{k-s} M_{m+sn} L_n + A_{1} L_n^k M_{m+n}\\
			= & \sum_{s=1}^{k} \binom{k}{s} A_{s} L_n^{k-s} ([M_{m+sn}, L_n] + L_n M_{m+sn}) + A_{1} L_n^k M_{m+n}\\
			= & \sum_{s=1}^{k} \binom{k}{s} A_{s+1} L_n^{k-s} M_{m+(s+1)n} + \sum_{s=1}^{k} \binom{k}{s} A_{s} L_n^{k-s+1} M_{m+sn}+ A_{1} L_n^k M_{m+n}\\
			= & \sum_{s=1}^{k} \binom{k}{s} A_{s+1} L_n^{k-s} M_{m+(s+1)n} + \sum_{s=2}^{k} \binom{k}{s} A_{s} L_n^{k-s+1} M_{m+sn} + (k+1)A_{1} L_n^k M_{m+n} \\
			= & \sum_{s=2}^{k+1} \binom{k}{s-1} A_{s} L_n^{k-s+1} M_{m+sn} + \sum_{s=2}^{k} \binom{k}{s} A_{s} L_n^{k-s+1} M_{m+sn} + (k+1)A_{1} L_n^k M_{m+n}.
		\end{align*}
		Using the property $\binom{k}{s} + \binom{k}{s-1} = \binom{k+1}{s}$, we obtain the form for $i=k+1$.
		The equation (\ref{MLYL2}) can be proved by a similar induction on $i$. Hence the lemma holds.
	\end{proof}
	
	\begin{remark}
		In Lemma \ref{LYM}, we can also write formulae (\ref{MLYL1}) and (\ref{MLYL2}) as follow:
		\begin{equation}\label{MLYL}
			\begin{split}
				[M_m,L_n^i]=&\,i(\lambda n-m-2\mu)L_n^{i-1} M_{m+n}+w_1,\\
				[Y_{m-\frac{1}{2}},L_n^i]=&\,i\Big(\frac{\lambda+1}{2}n+\frac{1}{2}-m-\mu\Big)L_n^{i-1} Y_{m+n-\frac{1}{2}}+w_2,
			\end{split}
		\end{equation}
		where $w_1, w_2 \in\mathcal{U}(\Glm)$ and for any $v\in V$
		$$\mathrm{deg}(w_1 v)\prec \mathrm{deg}(L_n^{i-1} M_{m+n}v),\ \ \mathrm{deg}(w_2 v)\prec \mathrm{deg}(L_n^{i-1} Y_{m+n-\frac{1}{2}}v).$$
	\end{remark}
	
	Now we can summarize the key result for induced  $\Glm$-modules as follows.
	
	\begin{theo}\label{th1}
		Let $d_1,d_2\in\N$ and $V$ be a simple $\Glm^{(d_1,d_2)}$-module. Suppose there exists  $t\in\N$  satisfying the following three conditions:
		
		\noindent
		{\rm (a)} the action of  $M_{t}$ on  $V$ is injective;
		
		\noindent
		{\rm (b)}   $M_iV=Y_{j-\frac{1}{2}}V=L_kV=0$ for  all $i>t,j>t+d_2$ and $k>t+d_1$;
		
		\noindent
		{\rm (c)}  $\big((\lambda+1)i+t+2\mu\big)\big(i+d_1+\lambda (i+t+d_1)-2 \mu\big)\neq 0$, $\forall\, i\in\Z_+$.
		
		\noindent
		Then we have
		
		\noindent
		{\rm (I)} $\mathrm{Ind}(V)$ is a simple $\Glm$-module;
		
		\noindent
		{\rm (II)} the actions of  $M_i,Y_{j-\frac{1}{2}},L_k$ on $\mathrm{Ind}(V)$  for  all $i>t,j>t+d_2$ and $k>t+d_1$  are locally nilpotent.
	\end{theo}
	
	\begin{proof}
		In order to prove (I) of Theorem \ref{th1}, we need the following claim.
		\begin{clai}\label{claim3.1}
			For any $v\in\mathrm{Ind}(V)\setminus V$, let $\mathrm{deg}(v)=(\mi,\mj,\mk)$,
			$\tilde{i}=\mathrm{max}\{s:i_s\neq0\}$ if $\mi\neq\mathbf{0}$, ${\hat{j}}=\mathrm{min}\{s:j_s\neq0\}$ if $\mj\neq\mathbf{0}$ and  $\hat{k}=\mathrm{min}\{s:k_s\neq0\}$ if $\mk\neq\mathbf{0}$.  Then we obtain
			
			\noindent
			{\rm (1)}   if $\mk\neq\mathbf{0}$,  then $\hat{k}>0$ and $\mathrm{deg}(M_{\hat{k}+t}v)=(\mi,\mj,\mk^{\prime\prime})$;
			
			\noindent
			{\rm (2)}  if $\mk=\mathbf{0},\mj\neq\mathbf{0}$,  then ${\hat{j}}>0$ and $\mathrm{deg}(Y_{{\hat{j}}+t+d_2-\frac{1}{2}}v)=(\mi,\mj^{\prime\prime},0)$;
			
			\noindent
			{\rm (3)} if $\mj=\mk=\mathbf{0},\mi\neq\mathbf{0}$,  then $\tilde{i}>0$ and $\mathrm{deg}(L_{\tilde{i}+t+d_1}v)=(\mi^{\prime},0,0)$.
		\end{clai}
		To prove this claim, we assume that $v$ is of the form in \eqref{def2.1}.
		
		{\rm (1)} It is enough to show what we want to have  by comparing the degree. Now we consider those $v_{\mx,\my,\mz}$
		with $$M_{\hat{k}+t}M^{\mx}Y^{\my}L^{\mz}v_{\mx,\my,\mz}\neq 0.$$
		Note that $M_{\hat{k}+t}v_{\mx,\my,\mz}=0$ and $Y_{\hat{k}+t+d_2-\frac{1}{2}}v_{\mx,\my,\mz}=0$ for any  $(\mx,\my,\mz)\in \mathrm{supp}(v)$. By (\ref{def1.1}), one can easily check that
		\begin{equation}\label{MYv}
			M_{\hat{k}+t}M^{\mx}Y^{\my}L^{\mz}v_{\mx,\my,\mz}=M^{\mx}Y^{\my}[M_{\hat{k}+t},L^{\mz}]v_{\mx,\my,\mz}.
		\end{equation}
		
		From (a) and (c), we have $M_{t}v_{\mx,\my,\mz}\neq0$.
		If $\mz=\mk$, then by Lemma \ref{LYM}, (\ref{MYv}) turns into
		\begin{equation}\label{MYv2}
			M_{\hat{k}+t}M^{\mx}Y^{\my}L^{\mz}v_{\mx,\my,\mz}=-k_{\hat{k}}\big( (\lambda+1)\hat{k}+t+2\mu\big)M^{\mx}Y^{\my}L^{\mk-\epsilon_{\hat{k}}}M_{t}v_{\mx,\my,\mk}+\mathrm{lower \  terms}.
		\end{equation}
		It is clear that
		$$\mathrm{deg}(M_{\hat{k}+t}M^{\mx}Y^{\my}L^{\mz}v_{\mx,\my,\mz})=
		(\mx,\my,\mk^{\prime\prime})\preceq (\mi,\mj,\mk^{\prime\prime}).$$

		Now we suppose  $(\mz,\mathbf{w}(\mz))\prec(\mk,\mathbf{w}(\mk))$ and denote
		\begin{equation*}
			\aligned
			\mathrm{deg}(M_{\hat{k}+t}M^{\mx}Y^{\my}L^{\mz}v_{\mx,\my,\mz})=(\mx_1,\my_1,\mz_1)\in \mathbf{S}\times \mathbf{S}\times \mathbf{S}.
			\endaligned
		\end{equation*}
		If $\mathbf{w}(\mz)<\mathbf{w}(\mk)$, then we get $\mathbf{w}(\mz_1)\leq \mathbf{w}(\mz)-\hat{k}<\mathbf{w}(\mk)-\hat{k}=\mathbf{w}(\mk^{\prime\prime})$,
		which gives rise to $(\mx_1,\my_1,\mz_1)\prec(\mi,\mj,\mk^{\prime\prime})$.
		
		Then we suppose $\mathbf{w}(\mz)=\mathbf{w}(\mk)$ and ${\mz}\prec{\mk}$. Let $\hat{z}:=\mathrm{min}\{s:z_s\neq0\}>0$.
		If $\hat{z}>\hat{k}$, it is easy to see that $\mathbf{w}(\mz_1)<\mathbf{w}(\mz)-\hat{k}=\mathbf{w}(\mk^{\prime\prime})$. If $\hat{z}=\hat{k}$,
		we can  deduce $(\mx_1,\my_1,\mz_1)=(\mx,\my,\mz^{\prime\prime})$.
		Since $\mz^{\prime\prime}\prec \mk^{\prime\prime}$, we have $\mathrm{deg}(M_{\hat{k}+t}M^{\mx}Y^{\my}L^{\mz}v_{\mx,\my,\mz})
		=(\mx_1,\my_1,\mz_1)\prec(\mi,\mj,\mk^{\prime\prime})$ in both cases.
		
		From above discussion, we conclude that
		$\mathrm{deg}(M_{\hat{k}+t}v)=(\mi,\mj,\mk^{\prime\prime})$.

		{\rm (2)} Now we use   a similar method that appeared  above.  We consider   $v_{\mx,\my,\mathbf{0}}$
		with $$Y_{{\hat{j}}+t+d_2-\frac{1}{2}}M^{\mx}Y^{\my}v_{\mx,\my,\mathbf{0}}\neq0.$$
		Since $Y_{{\hat{j}}+t+d_2-\frac{1}{2}}v_{\mx,\my,\mathbf{0}}=0$ for any  $(\mx,\my,\mathbf{0})\in \mathrm{supp}(v)$,
		then we have
		$$Y_{{\hat{j}}+t+d_2-\frac{1}{2}}M^{\mx}Y^{\my}v_{\mx,\my,\mathbf{0}}=
		M^{\mx}[Y_{{\hat{j}}+t+d_2-\frac{1}{2}},Y^{\my}]v_{\mx,\my,\mathbf{0}}.$$
		Note that $M_tv_{\mx,\my,\mathbf{0}}\neq0$.
		If $\my=\mj$, it is easy to get that
		$$\mathrm{deg}(Y_{{\hat{j}}+t+d_2-\frac{1}{2}}M^{\mx}Y^{\my}v_{\mx,\my,\mathbf{0}})=
		(\mx,\my^{\prime\prime},\mathbf{0})\preceq (\mi,\mj^{\prime\prime},\mathbf{0}),$$
		where the equality holds if and only if $\mx=\mi$.
		
		Suppose $(\my,\mathbf{w}(\my))\prec(\mj,\mathbf{w}(\mj))$, and we can write  $$\mathrm{deg}(Y_{{\hat{j}}+t+d_2-\frac{1}{2}}M^{\mx}Y^{\my}v_{\mx,\my,\mathbf{0}})
		=(\mx_1,\my_1,\mathbf{0})\in \mathbf{S}\times \mathbf{S}\times \mathbf{S}.$$ If $\mathbf{w}(\my)<\mathbf{w}(\mj)$,
		then we get $\mathbf{w}(\my_1)\leq \mathbf{w}(\my)-{\hat{j}}<\mathbf{w}(\mj)-{\hat{j}}=\mathbf{w}(\mj^{\prime\prime})$,
		which shows that $(\mx_1,\my_1,\mathbf{0})\prec(\mi,\mj^{\prime\prime},\mathbf{0})$.
		
		First, we assume $\mathbf{w}(\my)=\mathbf{w}(\mj)$ and ${\my}\prec{\mj}$. Let $\hat{y}:=\mathrm{min}\{s:y_s\neq0\}>0$.
		If $\hat{y}>{\hat{j}}$, we obtain $\mathbf{w}(\my_1)<\mathbf{w}(\my)-{\hat{j}}=\mathbf{w}(\mj^{\prime\prime})$. If $\hat{y}={\hat{j}}$,
		we can similarly check that $(\mx_1,\my_1,\mathbf{0})=(\mx,\my^{\prime\prime},\mathbf{0})$.
		By $\my^{\prime\prime}\prec \mj^{\prime\prime}$, we have $\mathrm{deg}(Y_{{\hat{j}}+t+d_2-\frac{1}{2}}M^{\mx}Y^{\my}v_{\mx,\my,\mathbf{0}})
		=(\mx_1,\my_1,\mathbf{0})\prec(\mi,\mj^{\prime\prime},\mathbf{0})$ in both cases.
		
		Consequently, we conclude that
		$\mathrm{deg}(Y_{{\hat{j}}+t+d_2-\frac{1}{2}}v)=(\mi,\mj^{\prime\prime},\mathbf{0})$.
		
		{\rm (3)} Noticing that $L_{\tilde{i}+t+d_1}V=M_{t+n}V=0$ for $n\in\Z_+$, denote $\tilde{x}:=\mathrm{max}\{s:x_s\neq0\}>0$, then from Lemma \ref{LYM},  we have
		$$L_{\tilde{i}+t+d_1}M^{\mx}v_{\mx,\mathbf{0},\mathbf{0}}=
		-x_{\tilde{x}}\big(\tilde{x}+d_1+\lambda (\tilde{i}+t+d_1)-2 \mu\big)M^{\mx^\prime}M_{t+\tilde{i}-\tilde{x}}v_{\mx,\mathbf{0},\mathbf{0}}.$$
		By the definition of $\tilde{i}$, we get $\tilde{x}\leq\tilde{i}$. Then from {\rm (c)} we have $\mathrm{deg}(L_{\tilde{i}+t+d_1}v)\preceq(\mi^{\prime},0,0)$, where the equality holds if and only if $\mx=\mi$. This proves Claim \ref{claim3.1}.
		
		Using Claim \ref{claim3.1} repeatedly, from any nonzero element $v\in\mathrm{Ind}(V)$ we can reach a nonzero element in
		$\mathcal{U}(\Glm)v\cap V\neq0$, which indicates the simplicity of $\mathrm{Ind}(V)$.

		Next, we prove (II) of Theorem \ref{th1}.
		We  consider those
		$v_{\mx,\my,\mz}$ in \eqref{def2.1} with
		$M_{i}M^{\mx}Y^{\my}L^{\mz}
		$ $v_{\mx,\my,\mz}\neq0,
		Y_{j-\frac{1}{2}}M^{\mx}Y^{\my}
		L^{\mz}v_{\mx,\my,\mz}\neq0$ and $
		L_{k}M^{\mx}
		Y^{\my}L^{\mz}v_{\mx,\my,\mz}\neq0$
		for $i>t,j>t+d_2$ and $k>t+d_1$, then $L^{\mz}\neq 0$, hence we can assume $\tilde{z}:=\mathrm{max}\{s:z_s\neq0\}>0$.
		For $M_{i}$, we have
		\begin{eqnarray*}
			M_{i}M^{\mx}Y^{\my}L^{\mz}v_{\mx,\my,\mz}
			&=&
			\sum_{{\alpha}=1}^{\tilde{z}} \sum_{{\beta}=1}^{z_\alpha }M^{\mx}Y^{\my}L_{-\tilde{z}}^{z_{\tilde{z}}} \cdots L_{-{\alpha}}^{z_{\alpha}-\beta} [M_i,L_{-{\alpha}}]L_{-{\alpha}}^{\beta-1}\cdots L_{-1}^{z_1}v_{\mx,\my,\mz}
			\\&=&
			\sum_{{\alpha}=1}^{\tilde{z}}\sum_{{\beta}=1}^{z_\alpha }-(i+\la \alpha+2\mu)M^{\mx}Y^{\my}L_{-\tilde{z}}^{z_{\tilde{z}}} \cdots L_{-{\alpha}}^{z_{\alpha}-\beta} M_{i-{\alpha}}L_{-{\alpha}}^{\beta-1}\cdots L_{-1}^{z_1}v_{\mx,\my,\mz}.
		\end{eqnarray*}
		Counting the sum of degrees of all $L_{-{\alpha}}$, where $\alpha=1,\ldots,\tilde{z}$, one can find it decreases from $\mathbf{d}(\mz)$ to $\mathbf{d}(\mz)-1$. Repeating the above step, we have $M_i^{\mathbf{d}(\mz)+1}M^{\mx}Y^{\my}L^{\mz}v_{\mx,\my,\mz}=0$.
		It is clear that $M_i$ acts locally
		nilpotently on $\mathrm{Ind}(V)$ for $i>t$. Similarly, we obtain that $Y_{j-\frac{1}{2}}^{\mathbf{d}(\my)+2\mathbf{d}(\mz)+1}M^{\mx}Y^{\my}L^{\mz}v_{\mx,\my,\mz}=0$, which means $Y_{j-\frac{1}{2}}$ acts
		locally nilpotently on $\mathrm{Ind}(V)$ for $j>t + d_2$. For $L_k$ and any $v\in V$, we first prove the following claim.
		\begin{clai}\label{claimnew}
			%$v\in V$, $\mz\in\mathbf{S}$ and
			For any $n\in \Z_+$, $
			L_{\mathbf{w}(\mz)+t+d_1+n}L^{\mz}v=
			Y_{\mathbf{w}(\mz)+t+d_2+n-\frac{1}{2}}L^{\mz}v=
			M_{\mathbf{w}(\mz)+t+n}L^{\mz}v=0$.
		\end{clai}
		
		We prove this by induction on $\mathbf{d}(\mz)$. For $\mathbf{d}(\mz)=1$, $L^{\mz}v$ turns into $L_{-\alpha}v$ for some $\alpha\in\Z_+$, then $\mathbf{w}(\mz)=\alpha$. By (\ref{def1.1}) and {\rm (b)}, we have
		\begin{eqnarray*}
			L_{\alpha+t+d_1+n}L_{-\alpha}v&=&([L_{\alpha+t+d_1+n},L_{-\alpha}]+L_{-\alpha}L_{\alpha+t+d_1+n})v\\
			&=&-(2\alpha+t+d_1+n)L_{t+d_1+n}v=0,\\
			Y_{\alpha+t+d_2+n-\frac{1}{2}}L_{-\alpha}v&=&([Y_{\alpha+t+d_2+n-\frac{1}{2}},L_{-\alpha}]+L_{-\alpha}Y_{\alpha+t+d_2+n-\frac{1}{2}})v\\
			&=&-(\frac{\la+3}{2}\alpha+t+d_2+n+\mu-\frac{1}{2})Y_{t+d_2+n-\frac{1}{2}}v=0,\\
			M_{\alpha+t+n}L_{-\alpha}v&=&([M_{\alpha+t+n},L_{-\alpha}]+L_{-\alpha}M_{\alpha+t+n})v\\
			&=&-\big((1+\la)\alpha+t+n+2\mu\big)M_{t+n}v=0.
		\end{eqnarray*}
		Now we assume that $\mathbf{d}(\mz)>1$ and the claim holds for $\mathbf{d}(\mz)-1$. Note that $\mathbf{d}(\mz^{\prime})=\mathbf{d}(\mz)-1$. Then for $L_{\mathbf{w}(\mz)+t+d_1+n}L^{\mz}v$, it follows that
		\begin{eqnarray*}
			L_{\mathbf{w}(\mz)+t+d_1+n}L^{\mz}v
			&=&L_{\mathbf{w}(\mz)+t+d_1+n}L_{-\tilde{z}}L^{\mz^{\prime}}v\\
			&=&([L_{\mathbf{w}(\mz)+t+d_1+n},L_{-\tilde{z}}]+L_{-\tilde{z}}L_{\mathbf{w}(\mz)+t+d_1+n})L^{\mz^{\prime}}v,\\
			&=&-(\tilde{z}+\mathbf{w}(\mz)+t+d_1+n)L_{\mathbf{w}(\mz^{\prime})+t+d_1+n}L^{\mz^{\prime}}v+L_{-\tilde{z}}L_{\mathbf{w}(\mz^{\prime})+t+d_1+n+\tilde{z}}L^{\mz^{\prime}}v.
		\end{eqnarray*}
		By the induction hypothesis, we know $L_{\mathbf{w}(\mz^{\prime})+t+d_1+n}L^{\mz^{\prime}}v=L_{\mathbf{w}(\mz^{\prime})+t+d_1+n+\tilde{z}}L^{\mz^{\prime}}v=0$, which leads to $L_{\mathbf{w}(\mz)+t+d_1+n}L^{\mz}v=0$. Similarly, we can obtain that $Y_{\mathbf{w}(\mz)+t+d_2+n-\frac{1}{2}}L^{\mz}v$ $=
		M_{\mathbf{w}(\mz)+t+n}$ $L^{\mz}v =0$. Therefore, Claim \ref{claimnew} has been proved.
		%Besides, from the proof above, it is not difficult to verify if $L_{n_1}L^{\mz}v=Y_{n_2-\frac{1}{2}}L^{\mz}v=M_{n_3}L^{\mz}v=0$, then $L_{n_1}L_{i}L^{\mz}v=Y_{n_2-\frac{1}{2}}L_{i}L^{\mz}v=M_{n_3}L_{i}L^{\mz}v=0$ for all $i\in\N$.
		
		Denote $\varphi:=\mathrm{ad}(L_k)$, then
		\begin{eqnarray}\label{N123}
			L_{k}^{N}M^{\mx}Y^{\my}L^{\mz}v_{\mx,\my,\mz}
			&=&
			\sum_{\substack{N_1,N_2,N_3\in\N,\\N_1+N_2+N_3=N}}\varphi^{N_1}(M^{\mx})\varphi^{N_2}(Y^{\my})\varphi^{N_3}(L^{\mz})v_{\mx,\my,\mz},
			%\\&=&
			%\sum_{{\alpha}=1}^{\tilde{z}}\sum_{{\beta}=1}^{z_\alpha }-(i+\la \alpha+2\mu)M^{\mx}Y^{\my}L_{-\tilde{z}}^{z_{\tilde{z}}} \ldots L_{-{\alpha}}^{z_{\alpha}-\beta} M_{i-{\alpha}}L_{-{\alpha}}^{\beta-1}\ldots L_{-1}^{z_1}v_{\mx,\my,\mz},
		\end{eqnarray}
		where $N\in\Z_+$. Taking $kN>\big(2\mathbf{w}(\mz)+t+d_1\big)\mathbf{d}(\mx)
		+\big(\mathbf{w}(\mathbf{y+z})+t+2d_2+1\big)\mathbf{d}(\my)
		+\big(\mathbf{w}(\mz)+t+d_1\big)\mathbf{d}(\mz)
		+\mathbf{w}(\mx+\my+\mz)$, for $N_1,N_2,N_3\in\N$ in (\ref{N123}), we can deduce that at least one of the following inequalities holds
		\begin{eqnarray}
			\label{N1}
			kN_1&>&\big(2\mathbf{w}(\mz)+t+d_1\big)\mathbf{d}(\mx)+\mathbf{w}(\mx),\\
			\label{N2}
			kN_2&>&\big(\mathbf{w}(\mathbf{y+z})+t+2d_2+1\big)\mathbf{d}(\my)+\mathbf{w}(\my),\\
			\label{N3}
			kN_3&>&\big(\mathbf{w}(\mz)+t+d_1\big)\mathbf{d}(\mz)+\mathbf{w}(\mz).
		\end{eqnarray}
		
		If (\ref{N3}) holds, we have
		$$\varphi^{N_3}(L^{\mz})v_{\mx,\my,\mz}
		=\varphi^{N_3}(L_{-\alpha_{\mathbf{d}(\mz)}}\cdots L_{-\alpha_2} L_{-\alpha_1})v_{\mx,\my,\mz}
		=L_{\beta_{\mathbf{d}(\mz)}}\cdots L_{\beta_2} L_{\beta_1}v_{\mx,\my,\mz},$$
		where $1\leq \alpha_1\leq \alpha_2\leq\cdots \leq\alpha_{\mathbf{d}(\mz)}$ and $\sum_{i=1}^{\mathbf{d}(\mz)}\alpha_{i}=\mathbf{w}(\mz)$. One can check
		
		$$\sum_{i=1}^{\mathbf{d}(\mz)}\beta_{i}=kN_3-\sum_{i=1}^{\mathbf{d}(\mz)}\alpha_{i}>\big(\mathbf{w}(\mz)+t+d_1\big)\mathbf{d}(\mz).$$
		Then, there exists $i'\in\{1,2,\cdots,\mathbf{d}(\mz)\}$ such that $L_{\beta_{i'}}>\mathbf{w}(\mz)+t+d_1$. It follows from Claim \ref{claimnew} that $\varphi^{N_3}(L^{\mz})v_{\mx,\my,\mz}
		=0$. Similarly, one has $\varphi^{N_1}(M^{\mx})\varphi^{N_2}(Y^{\my})\varphi^{N_3}(L^{\mz})v_{\mx,\my,\mz}=0$ if  (\ref{N1}) holds.
		
		If (\ref{N2}) holds, we can get
		\begin{eqnarray*}
			\varphi^{N_2}(Y^{\my})\varphi^{N_3}(L^{\mz})v_{\mx,\my,\mz}
			&=&\varphi^{N_2}(Y_{-\alpha_{\mathbf{d}(\my)}}\cdots Y_{-\alpha_2} Y_{-\alpha_1})\varphi^{N_3}(L^{\mz})v_{\mx,\my,\mz}\\
			&=&Y_{\beta_{\mathbf{d}(\my)}}\cdots Y_{\beta_2} Y_{\beta_1}\varphi^{N_3}(L^{\mz})v_{\mx,\my,\mz},
		\end{eqnarray*}
		where $d_2+\frac{1}{2}\leq \alpha_1\leq \alpha_2\leq\cdots \leq\alpha_{\mathbf{d}(\my)}$ and $\sum_{i=1}^{\mathbf{d}(\mz)}\alpha_{i}=\mathbf{w}(\my)+(d_2+\frac{1}{2})\mathbf{d}(\my)$. We obtain
		$$\sum_{i=1}^{\mathbf{d}(\my)}\beta_{i}=kN_2-\sum_{i=1}^{\mathbf{d}(\my)}\alpha_{i}>\big(\mathbf{w}(\mathbf{y+z})+t+d_2+\frac{1}{2}\big)\mathbf{d}(\my).$$
		Then, there exists $i'\in\{1,2,\cdots,\mathbf{d}(\my)\}$ such that $Y_{\beta_{i'}}>\mathbf{w}(\my+\mz)+t+d_2+\frac{1}{2}$. If $i'=1$, by Claim \ref{claimnew}, $\varphi^{N_2}(Y^{\my})\varphi^{N_3}(L^{\mz})v_{\mx,\my,\mz}
		=0$. If $i'\geq 2$,
		\begin{eqnarray*}
			Y_{\beta_{i'}}Y_{\beta_{i'-1}}\cdots  Y_{\beta_1}\varphi^{N_3}(L^{\mz})v_{\mx,\my,\mz}            &=&\sum_{j=1}^{i'-1}Y_{\beta_{i'-1}}\cdots[Y_{\beta_{i'}},Y_{\beta_{j}}]\cdots Y_{\beta_1}\varphi^{N_3}(L^{\mz})v_{\mx,\my,\mz}\\
			&=&\sum_{j=1}^{i'-1}(i'-j)Y_{\beta_{i'-1}}\cdots M_{\beta_{i'}+\beta_{j}}\cdots Y_{\beta_1}\varphi^{N_3}(L^{\mz})v_{\mx,\my,\mz}.
		\end{eqnarray*}
		Note that $\beta_{j}\geq-\alpha_{\mathbf{d}(\my)}\geq -\mathbf{w}(\mathbf{y})-d_2-\frac{1}{2}$. We can deduce that $\beta_{i'}+\beta_{j}> \mathbf{w}(\mathbf{z})+t$. By Claim \ref{claimnew}, $\varphi^{N_2}(Y^{\my})\varphi^{N_3}(L^{\mz})v_{\mx,\my,\mz}
		=0$.
		
		In conclusion, $L_k$ acts
		locally nilpotently on $\mathrm{Ind}(V)$ for $k >t + d_1$.
		The proof of this theorem is completed.
	\end{proof}
	
	\begin{rema}\label{rema3.3} \rm
		In Theorem \ref{th1},
		%when $t=0$, the condition ${\rm (a)}$ is equivalent to that $\nu_0\neq 0$. In addition,
		from the above proof, we see that Claim \ref{claim3.1} also holds without the assumption of the simplicity of $V$ as a $\Glm^{(d_1,d_2)}$-module.
	\end{rema}
	
	Moreover, we have the following corollary.
	\begin{coro}
		Letting $d_1,d_2$ and $V$ as in Theorem \ref{th1} except that $V$ may not be simple over $\Glm^{(d_1,d_2)}$, then we have
		$$V=\{v\in\mathrm{Ind}(V)\, |\, M_iv=Y_{j-\frac{1}{2}}v=L_kv=0,\quad \forall \, i>t,j>t+d_2,k>t+d_1\}.$$
	\end{coro}
	
	%\begin{rema}\label{rema3.4}
	%Note that in  Theorem \ref{th1}, $t$ is the largest integer among all $l$ such that $M_l$ acts nontrivially on $V$, while $t+d$ is not necessarily the largest integer among all $l$ such that $Y_{l-\frac{1}{2}},L_l$ acts nontrivially on $V$.
	%If we take $d=0$ in the theorem \ref{th1}, the resulted modules for $t=0$ are simple highest weight modules,
	%which reobtains the simplicity of the Verma modules in the \cite{TZ} for a special case.
	%\end{rema}
	Recall the \emph{Virasoro algebra} is $\mathrm{Vir}:=\mathrm{Span}_{\C}\{L_m,C \,|\, m\in\Z\}$ with the following brackets
	$$[L_m,L_n]= (n-m)L_{m+n}+\delta_{m+n,0}\frac{m^{3}-m}{12}C,\quad [L_m,C]=0,\quad \forall\, m,n\in\Z.$$
	To prove the next theorem, we first introduce the following lemma.
	\begin{lemm}[\cite{MNTZ}]\label{le3.7}
		Let $V$ be a $\mathrm{Vir}$-module such that $L_{k}$ acts locally finitely on $V$ for some $k \in \Z_+$. Then, there exist $0 \neq v \in V$ and $N \in \Z_+$ such that
		$$
		L_{n} v=0,\quad \forall \, \,  n \geq N.
		$$
	\end{lemm}
	
	Now we are ready to state the second main result of this section,  which improves the results over  the Schr\"{o}dinger-Virasoro algebra (see  \cite[Theorem 3.4]{CHS} ) and  $\mathfrak{bms}_3$ algebra (see \cite[Theorem 2]{Cqf}). Noting that in (1) and (2) of the following theorem, we only require the locally finite and locally nilpotent conditions of $L_t$ on $P$.
	\begin{theo}\label{th2}
		Let $P$ be a simple  $\Glm$-module. The following conditions are equivalent:
		
		\noindent
		{\rm (1)}  There exists $t\in\Z_+$ such that the action of $L_t$ on $P$ is locally finite.
		
		\noindent
		{\rm (2)}  There exists $t\in\Z_+$ such that the action of $L_t$ on $P$ is locally nilpotent.
		
		\noindent
		{\rm (3)}  There exist $x,y,z\in\Z$ such that  $P$ is a  locally finite $\Glm^{(x,y,z)}$-module.
		
		\noindent
		{\rm (4)}    There exist $x,y,z\in\Z$ such that  $P$ is a  locally nilpotent $\Glm^{(x,y,z)}$-module.
		
		\noindent
		{\rm (5)} $P$ is restricted.
	\end{theo}

	\begin{proof}
		One can check that $(5)\Rightarrow(3)\Rightarrow(1)$, $(5)\Rightarrow(4)\Rightarrow(2)$
		and $(2)\Rightarrow(1)$ are clear. Hence we only need to prove  $(1)\Rightarrow(5)$. Since $L_t$ acts locally finitely on $P$,    there exists a nonzero $v \in P$ and $\omega \in \C$ such that $L_t v = \omega v$. By Lemma \ref{le3.7}, one can see that there exists $N_1 \in \Z_+$ such that $L_n v = 0$ for all $n \geq N_1$.
		
		\begin{clai}\label{claim3}
			For any $j\in\{1, \dots, t\}$, there exists $\Sj\in\Z_+$ such that $j+(s-\la) t +2\mu \neq 0$ and $j+(s-\frac{\la+1}{2})t+\mu-\frac{1}{2} \neq 0$ for all $s\in \Z_+$ and $s\geq \Sj$.
		\end{clai}
		Take $j\in\{1, \dots, t\}$. If for all $s\in \Z_+$, $j+(s-\la) t +2\mu \neq 0$ and $j+(s-\frac{\la+1}{2})t+\mu-\frac{1}{2} \neq 0$, then the claim holds trivially. Otherwise, if there exists $s_1\in \Z_+$ such that $j+(s_1-\la) t +2\mu = 0$, then for $s\geq s_1+1$:
		$$j+(s-\la) t +2\mu = j+(s_1-\la) t +2\mu +(s-s_1) t= (s-s_1) t\neq 0.$$
		Similarly,  if there exists $s_2\in \Z_+$ such that $j+(s_2-\frac{\la+1}{2})t+\mu-\frac{1}{2} = 0$, then $j+(s-\frac{\la+1}{2})t+\mu-\frac{1}{2} \neq 0$ for all $s\geq s_2+1$. Let $\Sj=\mathrm{max}\{s_1,s_2\}+1$. The claim is obtained.

		Then, we show that $M_n v = 0$ for sufficiently large $n$. Fix $j \in \{1, \dots, t\}$. Since $L_t$ is locally finite on $P$, the subspace $\mathcal{U}(\C L_t) M_{j+\Sj t} v$ is finite-dimensional. By Claim \ref{claim3} and the equation
		\begin{equation*}
			(L_t - \omega) M_{j+st} v = [L_t, M_{j+st}] v = \big(j+(s-\la)t + 2\mu\big) M_{j+(s+1)t} v,\quad \forall\, s\in\Z,
		\end{equation*}
		we obtain $\sum_{s \in \N} \C M_{j+(s+\Sj) t}v=\mathcal{U}(\C L_t) M_{j+\Sj} v$ is finite-dimensional, which means there exists a smallest $n_j\in\N$ such that
		$$M_{j+s_j t}v,\,M_{j+(s_j+1)t}v,\,\ldots,\,M_{j+(s_j+n_j)t}v$$
		are linearly dependent for some integer $s_j\geq\Sj$. Hence, there exists a polynomial $p(x)$ of degree $n_j$ such that $p(L_t) M_{j+s_j t} v = 0$.
		Let $p(x) = \sum_{i=0}^{n_j} a_i (x-\omega)^i$ with $a_i\in\C\ (i=1,2,\ldots,n_j)$ and $a_{n_j} \neq 0$. By Claim \ref{claim3}, applying $L_t-\omega$ to the equation $p(L_t) M_{j+s_j t} v = 0$ repeatedly, we can conclude that
		\begin{equation}\label{pLtMs}
			p(L_t) M_{j+s t} v = 0,\quad \forall\, s\in\Z\ \text{and}\ s\geq s_j.
		\end{equation}
		For convenience, we denote $B_{n}(s) = \prod_{i=0}^{n-1}\big(j+(s+i- \lambda)t + 2\mu\big)$ for  $s\in\Z,n\geq 1$ and $B_{0}(s)=1$, then (\ref{pLtMs}) turns into
		\begin{equation}\label{pLtMs2}
			\sum_{i=0}^{n_j} a_i B_{i}(s) M_{j+(s+i) t} v = 0,\quad \forall\, s\in\Z\ \text{and}\ s\geq s_j.
		\end{equation}
		For any integer $s\geq \mathrm{max}\{N_1,s_j\}$, applying $L_{st}$ to (\ref{pLtMs2}), one can get
		\begin{equation}\label{pLtMs3}
			0 =\sum_{i=0}^{n_j} a_i B_{i}(s) L_{st} M_{j+(s+i) t} v= \sum_{i=0}^{n_j} a_i B_{i}(s) \big(j+(s+i-\la s)t + 2\mu\big) M_{j+(2s+i) t} v.
		\end{equation}
		Since $s\geq \mathrm{max}\{N_1,s_j\}>0$, we have $2s>s>s_j$. We replace $s$ with $2s$ in (\ref{pLtMs2}):
		\begin{equation}\label{pLtMs4}
			\sum_{i=0}^{n_j} a_i B_{i}(2s) M_{j+(2s+i) t} v = 0.
		\end{equation}
		
		Suppose $n_j\geq 1$. Evaluating $(\ref{pLtMs3})+\big(j+(1-\la )st + 2\mu\big) \times (\ref{pLtMs4})$, we can deduce that
		\begin{equation}\label{aBsM}
			\sum_{i=1}^{n_j} a_i \Big(B_{i}(s) \big(j+(s+i-\la s)t + 2\mu\big)-B_{i}(2s)\big(j+(1-\la )st + 2\mu\big)\Big) M_{j+(2s+i) t} v = 0.
		\end{equation}
		Denote the coefficient of $M_{j+(2s+n_j) t}$ in (\ref{aBsM}) by  $q(s)$, then we obtain
		\begin{equation*}
			q(s)= a_{n_j} \Big(B_{n_j}(s) \big(j+(s+n_j-\la s)t + 2\mu\big)-B_{n_j}(2s)\big(j+(1-\la )st + 2\mu\big)\Big).
		\end{equation*}
		Now we prove there exists integer $s'\geq \mathrm{max}\{N_1,s_j\}$ such that $q(s')\neq 0$. Otherwise, $q(s)=0$ for all integers $s\geq \mathrm{max}\{N_1,s_j\}$, then polynomial $q(x)\equiv 0$. Hence, we only need to prove $q(x)\not\equiv 0$. If $\la\neq 1$,  one can check the coefficient of $x^{n_j+1}$ in $q(x)$
		\begin{equation*}
			a_{n_j}(1-\la)t^{n_j+1} (1-2^{n_j})\neq 0,
		\end{equation*}
		since $n_j\geq 1$ and $a_{n_j} \neq 0$. If  $\la=1$, the constant term and the coefficient of $x^{n_j}$ in $q(x)$ are:
		\begin{eqnarray*}
			q_{0}&=&a_{n_j} n_j \prod_{i=0}^{n_j-1}\big(j+(i- 1)t + 2\mu\big),\\
			q_{n_j}&=&a_{n_j}  t^{n_j}\big(j+n_j t+2\mu-2^{n_j}(j+2\mu)\big),
		\end{eqnarray*}
		respectively. If $q_{0}=q_{n_j}=0$, we have
		\begin{eqnarray}
			\label{q1}
			\prod_{i=0}^{n_j-1}\big(j+(i- 1)t + 2\mu\big)&=&0,\\
			\label{q2}
			j+n_j t+2\mu-2^{n_j}(j+2\mu)&=&0.
		\end{eqnarray}
		By Remark \ref{rema1}, $2\mu\in[-2,0)$. Since $j \in \{1, \dots, t\}$ and $t\in\Z_+$, we can get $\Re(j-t + 2\mu)<0$ and $\Re\big(j+(i- 1)t + 2\mu\big)>0$ for all integers $i\geq 3$. From (\ref{q1}), it is clear that
		\begin{equation}\label{jmu}
			j+2\mu=0\ \text{or}\ j+t + 2\mu=0.
		\end{equation}
		Substituting (\ref{jmu}) into (\ref{q2}) yields
		\begin{equation*}
			n_j t=0\ \text{or}\ (n_j+2^{n_j}-1)t=0.
		\end{equation*}
		This is a contradiction to $n_j\geq 1$ and $t\in\Z_+$.
		Combining the arguments above, we can conclude that there exists integer $s'\geq \mathrm{max}\{N_1,s_j\}$ such that $q(s')\neq 0$. It implies that
		$$
		M_{j+(2s'+1) t}v,\,M_{j+(2s'+2) t}v,\,\ldots,\,M_{j+(2s'+n_j) t}v
		$$
		are linearly dependent, which contradicts the minimality of $n_j$ unless $n_j = 0$. Thus (\ref{pLtMs}) reduces to $M_{j+s t} v = 0$ for any $s\in\Z$ and $s\geq s_j$. Taking $N_2=(\mathrm{max}\{s_1,s_2,\ldots,s_t\}+1)t\in\Z_+$, then $M_n v = 0$ for all $n \geq N_2$. Similarly, there exists $N_3\in\Z_+$ such that $Y_{n-\frac{1}{2}} v = 0$ for all $n \geq N_3$.

		Let $N = \max\{N_1, N_2, N_3\}$. We have shown that %$\Glm^{(N,N,N)}$ annihilates $v$, i.e.,
		$\Glm^{(n)}v=0$ for all $n\geq N$. Since $P$ is simple, one can get $P = \mathcal{U}(\Glm) v$. By the PBW theorem, any element $u$ in $ P$ is a linear combination of the finitely nonzero vectors
		$$M_{i_1} M_{i_2}\ldots  M_{i_p} Y_{j_1} Y_{j_2}\ldots  Y_{j_q}  L_{k_1} L_{k_2}\ldots L_{k_r}v,$$
		where $p,q,r\in\N$, $i_1,\ldots,i_p,k_1,\ldots,k_r\in\Z$ and $j_1,\ldots,j_q\in\Z+\frac{1}{2}$. For any $u \in P$ with above form, noting that $[\Glm^{(m)},\Glm^{(n)}]\subseteq \Glm^{(m+n)}$ for all $m,n\in\frac{1}{2}\Z$, we can take $T=N+\sum_{a=1}^{p}|i_a|+\sum_{b=1}^{q}|j_b|+\sum_{d=1}^{r}|k_d|$. Then for any $m\geq T$, $\Glm^{(m)}u=0$. Thus $P$ is restricted.
	\end{proof}
	
	Denote
	\begin{equation*}
		\aligned
		\Gamma_1=\Big\{(\la,\mu)\,|\,&\big((\lambda+1)i+t+2\mu\big)\big(i+\lambda (i+t)-2 \mu\big)\neq 0,\ \forall \,t\in\N\textrm{ and }i\in\Z_+\Big\}
		\endaligned
	\end{equation*}
	and
	\begin{equation*}
		\aligned
		\Gamma_2=\Big\{(\la,\mu)\,|\,&n-\lambda m+2\mu \neq 0,\  \forall \,m, n\in\Z\Big\}.
		\endaligned
	\end{equation*}
	It is clear that $\Gamma_2 \subset \Gamma_1$. The following result determines some simple restricted modules over $\Glm$.
	
	\begin{theo}\label{th3}
		Let $P$ be a simple restricted $\Glm$-module. Assume that  $(\la,\mu)\in\Gamma_1$ and there exists $a\in\Z_+$ such that the action of $M_a$ on $P$ is injective. Then there exist $d_1,d_2\in\N$ and a simple $\Glm^{(d_1,d_2)}$-module $V$ satisfying the conditions in Theorem \ref{th1} such that
		$P\cong\mathrm{Ind}(V)$.
	\end{theo}

	\begin{proof}
		For any $i,j,k\in \Z$, we consider the vector space
		$$F_{i,j,k}=\{v\in P\,|\, M_{i+n}v=(1-\delta_{j+n,0})Y_{j+n-\frac{1}{2}}v=L_{k+n}v=0, \quad \mbox{for\ all}\ n\in \Z_+\}.$$
		Since $P$ is a restricted $\Glm$-module, we have  $F_{i,j,k}\neq0$ for sufficiently large ${i,j},k\in\Z$. Noting that $M_a$ on $P$ is  injective, we can deduce that $F_{i,j,k}=0$ for all $i< a $.
		%On the other hand, $F_{i,j,k}=0$ for all $i<0$ since we have $M_0v=\nu_0 v\neq0$ for any nonzero $v\in S$.
		Thus we can find a smallest nonnegative integer $r_1$, and take the  integers $r_2,r_3\geq r_1$ with $r_3-r_1\geq2(r_2-r_1)-1$ such that  $F_{r_1,r_2,r_3}\neq 0$. Denote  $d_1=r_3-r_1, d_2=r_2-r_1$ and  $V=F_{r_1,r_2,r_3}$.
		For any $v\in V,i>r_1,j>r_2,k>r_3$ and $l\geq 1$, it follows from $k+l-d_2-\frac{1}{2}>r_3+\frac{1}{2}-d_2\geq r_2-\frac{1}{2}$ and $j+l-d_2-1>r_2-d_2=r_1$  that
		$$L_{k}(Y_{l-d_2-\frac{1}{2}}v)=(l-d_2-\frac{\lambda+1}{2}k+\mu-\frac{1}{2})Y_{k+l-d_2-\frac{1}{2}}v=0$$
		and
		$$
		Y_{j-\frac{1}{2}}(Y_{l-d_2-\frac{1}{2}}v)=(j-l-d_2)M_{j+l-d_2-1}v=0,
		$$
		respectively. Besides, for $i>r_1$, one has $
		M_{i}(Y_{l-d_2-\frac{1}{2}}v)=Y_{l-d_2-\frac{1}{2}}M_{i}v=0
		$. Thus we have $Y_{l-d_2-\frac{1}{2}}v\in V$ for
		all $l\geq 1$. Similarly, we can also obtain $M_{e-d_1}v\in V$ and $L_{e}v\in V$
		for all $e\in \N$. Therefore, $V$ is a $\Glm^{(d_1,d_2)}$-module.
		
		By the definition of $V$, we can obtain that the action of  $M_{r_1}$ on $V$ is injective. Since $P$ is simple
		and generated by $V$, then there exists a canonical surjective map
		$$\pi:\mathrm{ Ind}(V) \rightarrow P, \quad \pi(1\otimes v)=v,\quad \forall\,  v\in V.$$
		Next we only need to show that $\pi$ is also injective, that is to say, $\pi$ as the canonical map is bijective.  Let $K=\mathrm{ker}(\pi)$. Obviously, $K\cap V=0$. If	$K\neq0$, we can choose a nonzero vector $v\in K\setminus V$ such that $\mathrm{deg}(v)=(\mi,\mj,\mk)$ is minimal possible.
		Note that $K$ is a $\Glm$-submodule of $\mathrm{Ind}(V)$.
		Since $(\la,\mu)\in\Gamma_1$, by Claim \ref{claim3.1} and Remark \ref{rema3.3}, we can create a new vector $u\in K$  with $\mathrm{deg}(u)\prec(\mi,\mj,\mk)$, which is a contradiction. This forces $K=0$,
		that is, $P\cong \mathrm{Ind}(V)$. According to  the property of induced modules, we see that $V$ is simple as a $\Glm^{(d_1,d_2)}$-module.
	\end{proof}

	\begin{remark}
		In fact, we only need the condition $\Re(\mu)\in[-1,0)$ if $\la=1$. In particular, if $(\la,\mu)=(0,0)$, one can still get the above results.
	\end{remark}

	%\begin{remark} \rm
	%From the above proof, we know that any simple module satisfying conditions in
	%Theorem \ref{th2} is determined by some simple module $V$ over a certain subalgebra $\Glm^{(d_1,d_2)}$.
	%The conditions of Theorem  \ref{th1}  imply that  $V$  can be viewed as
	%a simple module over some finite-dimensional solvable quotient algebra of $\Glm^{(d_1,d_2)}$. This reduces the study of such modules over $\Glm$ to the study of simple modules over the corresponding finite-dimensional algebras.
	%\end{remark}
	%
	\section{Vertex algebras associated to $\Glm$}\label{Sec4}
	
	In this section, we  mainly  recall the definition of  vertex algebras associated to the Lie algebra $\Glm$ and some known  results (see \cite{LS}). Then  we give some applications of the weak modules of vertex   algebras.
	
	Let $W$ be a general vector space. Set
	
	$$
	\mathcal{E}(W)=\mathrm{Hom}\Big(W, W\big((x)\big)\Big) \subset \big(\mathrm{End} (W)\big)[[x, x^{-1}]],
	$$
	where
	$$
	W\big((x)\big)=\left\{ \sum_{n \in \Z} w_{n} x^{n} \mid w_{n} \in W, \, w_{n}=0 \text{ for } n \text{ sufficiently negative} \right\}.
	$$
	The identity element on $W$, denoted by $\mathbf{1}$, is a special element of $\mathcal{E}(W)$. Now, we give the definition of the vertex operator algebra and its (weak) module.
	
	\begin{defi}\rm
		A \emph{vertex algebra} denoted by a quadruple $(V,Y,\mathbf{1},D)$ is a vector space $V$ equipped with a linear map
		\begin{eqnarray*}
			Y(\cdot,z)\ :\ V&\rightarrow&\big(\mathrm{End}(V)\big)[[z,z^{-1}]],\\
			v&\mapsto&Y(v,z)=\sum_{n\in\Z}v_n z^{-n-1}\quad (\mathrm{where}\  v_n \in \mathrm{End}(V)).
		\end{eqnarray*}
		The vector $\mathbf{1}\in V$ (called the \emph{vacuum vector}) and the endomorphism $D$ of $V$ satisfy the following relations:
		
		\noindent
		{\rm (1)} For any $u,v \in V$, $u_nv =0$ for $n$ sufficiently large;
		
		\noindent
		{\rm (2)} $[D, Y(v, z)]= Y\big(D(v), z\big)=\frac{\mathrm{d}}{\mathrm{d}z}Y (v, z)$ for any $v \in V$;
		
		\noindent
		{\rm (3)} $Y(\mathbf{1},z) =\mathrm{Id}_V$ (called the \emph{identity operator} of $V$);
		
		\noindent
		{\rm (4)} $Y(v, z)\mathbf{1} \in \mathrm{End}(V)[[z]]$ and $\mathrm{lim}_{z\rightarrow 0}Y (v, z)\mathbf{1} = v$ for any $v \in V$;
		
		\noindent
		{\rm (5)} $z_{0}^{-1} \delta\left(\frac{z_{1}-z_{2}}{z_{0}}\right) Y(u, z_{1}) Y(v, z_{2}) -z_{0}^{-1} \delta\left(\frac{-z_{2}+z_{1}}{z_{0}}\right) Y(v, z_{2}) Y(u, z_{1}) = z_{2}^{-1} \delta\left(\frac{z_{1}-z_{0}}{z_{2}}\right) Y\big(Y(u, z_{0}) $ $v, z_{2}\big)$ for any $u,v \in V$ (the Jacobi identity).
	\end{defi}
	Besides, a vertex algebra $V$ is called a \emph{vertex operator algebra} if there exists another vector $\omega$ (called \emph{conformal vector}) of $V$ satisfying the following conditions:
	
	\noindent
	{\rm (6)} $[L(m), L(n)] = (n - m) L(m+n) + \frac{n^{3}-n}{12} \delta_{m+n, 0} C$ for $m, n \in \Z$, where $Y(\omega, z) = \sum_{n \in \Z} L(n)$ $ z^{-n-2}$;
	
	\noindent
	{\rm (7)} $L_{-1} = D$, i.e., $\frac{d}{dz} Y(v, z) = Y(L_{-1} v, z)$ for any $v \in V$;
	
	\noindent
	{\rm (8)} $V$ is $\Z$-graded such that $V = \oplus_{n \in \Z} V_{(n)}$, $L(0)|_{V_{(n)}} = n \mathrm{Id}_{V_{(n)}}$, $\dim(V_{(n)}) < \infty$ and $V_{(n)} = 0$ for $n$ sufficiently negative.
	
	Let $V$ be a vector space. A \emph{vertex operator} on $V$ is a formal series $a(z) = \sum_{n\in \Z} a_nz^{-n-1} \in \mathrm{End}(V)[[z,z^{-1}]]$ such that for any $u \in V, a_n u = 0$ for sufficiently large $n$, where $a\in V$. All vertex operators on $V$ form a vector space (over $\C$), denoted by $\mathrm{VO}(V)$. On $\mathrm{VO}(V)$, we have a linear endomorphism $D =\frac{\mathrm{d}}{\mathrm{d}z}$, the formal differentiation.
	
	Two vertex operators $a(z)$ and $b(z)$ on $V$ are said to be \emph{mutually local} if there is a non-negative integer $N$ such that
	$$(z_1-z_2)^{N}a(z_1) b(z_2) = (z_1-z_2)^{N} b(z_2) a(z_1).$$
	A space $S$, which consists of vertex operators, is said to be \emph{local} if any two vertex operators of $S$ are mutually local, and a maximal local space of vertex operators is called a \emph{local system}.
	
	\begin{defi}\label{VAmod}\rm
		Let $V$ be a vertex algebra. A $V$-\emph{module} is a triple $(W,d,Y_W)$ where $W$ is a vector space, $d$ is an endomorphism of $W$ and $Y_W$ is a linear map satisfying
		\begin{eqnarray*}
			Y_W(\cdot,z)\ :\ V&\rightarrow&\big(\mathrm{End}(W)\big)[[z,z^{-1}]],\\
			v&\mapsto&Y_W(v,z)=\sum_{n\in\Z}v_n z^{-n-1}\quad (\mathrm{where}\  v_n \in \mathrm{End}(W))
		\end{eqnarray*}
		and
		
		\noindent
		{\rm (1)} For any $v \in V$ and $u \in W$, $v_n u =0$ for $n$ sufficiently large;
		
		\noindent
		{\rm (2)} $[d, Y_W(v, z)]= Y_W\big(D(v), z\big)=\frac{\mathrm{d}}{\mathrm{d}z}Y_W (v, z)$ for any $v \in V$;
		
		\noindent
		{\rm (3)} $Y_W(\mathbf{1},z) =\mathrm{Id}_W$;
		
		\noindent
		{\rm (4)} $z_{0}^{-1} \delta\left(\frac{z_{1}-z_{2}}{z_{0}}\right) Y_W(u, z_{1}) Y_W(v, z_{2}) -z_{0}^{-1} \delta\left(\frac{-z_{2}+z_{1}}{z_{0}}\right) Y_W(v, z_{2}) Y_W(u, z_{1}) = z_{2}^{-1} \delta\left(\frac{z_{1}-z_{0}}{z_{2}}\right)Y_W\big( $ $Y_W(u, z_{0}) v$ $, z_{2}\big)$ for any $u,v \in V$.
		
		If $V$ is a vertex operator algebra, a module $W$ for $V$ as a vertex algebra is called a \emph{weak module} for $V$ as a vertex operator algebra.
	\end{defi}
	
	For $\Glm$, form three generating functions:
	$$L(x)=\sum _{n\in \mathbb{Z}}L_{n}x^{-n-2}, \quad \tilde{Y}(x)=\sum _{n\in \mathbb{Z}}Y_{n-\frac{1}{2}}x^{-n-1}, \quad M(x)=\sum _{n\in \mathbb{Z}}M_{n}x^{-n-1}.$$
	Then, Lie bracket relations in (\ref{def1.1}) can be written as (cf.  \cite{LS})
	%\begin{equation}\label{eqt4.1}
	%    \aligned
	\begin{align*}
		{}[L(x_{1}), L(x_{2})]=&-L'(x_{2}) x_{1}^{-1} \delta\left(\frac{x_{2}}{x_{1}}\right)-2 L(x_{2}) \frac{\partial}{\partial x_{2}} x_{1}^{-1} \delta\left(\frac{x_{2}}{x_{1}}\right)+\frac{1}{12}(\frac{\partial}{\partial x_{2}})^{3} x_{1}^{-1} \delta\left(\frac{x_{2}}{x_{1}}\right) C,\\
		{}[L(x_{1}), \tilde{Y}(x_{2})]= &-\tilde{Y}'(x_{2}) x_{1}^{-1} \delta\left(\frac{x_{2}}{x_{1}}\right)-\frac{1}{2}(\la+3) \tilde{Y}(x_{2}) \frac{\partial}{\partial x_{2}} x_{1}^{-1} \delta\left(\frac{x_{2}}{x_{1}}\right)\\
		&+\frac{1}{2}(2\mu+\la) x_{2}^{-1} \tilde{Y}(x_{2}) x_{1}^{-1} \delta\left(\frac{x_{2}}{x_{1}}\right),\\
		{}[\tilde{Y}(x_{1}), \tilde{Y}(x_{2})]=&M'(x_{2}) x_{1}^{-1} \delta\left(\frac{x_{2}}{x_{1}}\right)+2 M(x_{2}) \frac{\partial}{\partial x_{2}} x_{1}^{-1} \delta\left(\frac{x_{2}}{x_{1}}\right),\\
		{}[L(x_{1}), M(x_{2})]=&-M'(x_{2}) x_{1}^{-1} \delta\left(\frac{x_{2}}{x_{1}}\right)-(\la+1) M(x_{2}) \frac{\partial}{\partial x_{2}} x_{1}^{-1} \delta\left(\frac{x_{2}}{x_{1}}\right) \\
		&+(2\mu+\la) x_{2}^{-1} M(x_{2}) x_{1}^{-1} \delta\left(\frac{x_{2}}{x_{1}}\right).
	\end{align*}
	%\endaligned
	%\end{equation}
	
	Now, we consider the case $\mu+\frac{\la}{2}= 0$. For convenience, we denote $\G:=\G_{\la,-\frac{\la}{2}}$. It is clear that $\G$ is a $\frac{1}{2}\Z$-graded Lie algebra $\G=\oplus_{r \in \frac{1}{2}\Z} \G_{(r)}$ with
	\begin{eqnarray*}
		&&\G_{(0)}=\C L_{0} \oplus \C M_{0} \oplus \C C, \\
		&&\G_{(m-\frac{1}{2})}=\C Y_{m-\frac{1}{2}}, \\
		&&\G_{(n)}=\C L_{n} \oplus \C M_{n},
	\end{eqnarray*}
	where $m,n\in\Z$ and $n\neq 0$.
	Since $[L_{0}, M_{n}]=(n+2\mu) M_{n}$ and $[L_{0}, Y_{m-\frac{1}{2}}]=(m+\mu-\frac{1}{2}) Y_{m-\frac{1}{2}}$, this grading is not given by $\text{ad} L_{0}$-eigenvalues unless $\mu=0$. Let
	\begin{eqnarray*}
		&&\G_{+}=\mathrm{Span}_{\C}\{L_{n-1},Y_{n-\frac{1}{2}},M_{n},C\, |\, n\in\N \}, \\
		&&\G_{-}=\mathrm{Span}_{\C}\{L_{-n-1},Y_{-n-\frac{1}{2}},M_{-n}\, |\, n\in\Z_+ \}.
	\end{eqnarray*}
	
	Since we always assume that the action of $C$ is scalar $c\in\C$, form the induced module
	$$V(c)=\mathcal{U}(\G) \otimes_{\mathcal{U}(\G_{+})} \C. $$
	Identify $\C$ as a subspace of $V(c)$ and set $\mathbf{1}=1 \in \C \subset V(c)$. Then by the $\mathrm{PBW}$ Theorem, we have $V(c)=\oplus_{n \geq 0} V(c)_{(n)}$ with $V(c)_{(0)}=\C$, and $V(c)_{(n)}$ ($n \geq 1$) has a basis consisting of vectors
	$$
	M_{-n_{1}} \cdots M_{-n_{p}} Y_{-k_{1}-\frac{1}{2}} \cdots Y_{-k_{s}-\frac{1}{2}} L_{-m_{1}} \cdots L_{-m_{r}} \mathbf{1},
	$$
	where $r, s, p \geq 0$, $m_{1} \geq \cdots \geq m_{r} \geq 2$, $k_{1} \geq \cdots \geq k_{s} \geq 1$, $n_{1} \geq \cdots \geq n_{p} \geq 1$ with $\sum_{i=1}^{r} m_{i}+\sum_{j=1}^{s} k_{j}+\sum_{t=1}^{p} n_{t}=n$. One can check that $V(c)$ is a vertex algebra, which is uniquely determined by the condition that $\mathbf{1}$ is the vacuum vector and $Y(L_{-2} \mathbf{1}, x)=L(x)$, $Y(Y_{-\frac{3}{2}} \mathbf{1}, x)=\tilde{Y}(x)$ and  $Y(M_{-1} \mathbf{1}, x)=M(x)$ by the Theorem 5.7.1 in \cite{LL}. Then, we have the relation between restricted modules and vertex algebra modules by \cite{LS}.
	
	\begin{theo}[\cite{LS}, Proposition 2.14]\label{thm4.3}
		Fix $c \in \C$ and $\lambda \in \C$. There is a one-to-one correspondence between restricted $\G$-module $W$ of level $c$ and module $(W,L_{-1},Y_W)$ of the vertex algebra $V(c)$ with
		\begin{equation}\label{YW}
			Y_W(L_{-2} \mathbf{1}, x)=L(x),\quad Y_W(Y_{-\frac{3}{2}} \mathbf{1}, x)=\tilde{Y}(x),\quad Y_W(M_{-1} \mathbf{1}, x)=M(x).
		\end{equation}
	\end{theo}
	\begin{remark}
		Similarly, one can get the restricted $\G$-module $W$ structure from $V(c)$-module $(W,L_{-1},Y_W)$ still by (\ref{YW}).
	\end{remark}
	By Theorems \ref{th3} and   \ref{thm4.3}, we immediately get the following result.
	\begin{prop}
		Let $c\in\mathbb{C}$.  Then  simple restricted  $\mathcal{G}_{\lambda,\mu}$-modules given in Theorem  \ref{th3}  are precisely  simple  $V(c)$-modules
		for  $\mu+\frac{\la}{2}= 0$.
	\end{prop}
	Moreover, taking the conformal vector $\omega=L_{-2} \mathbf{1}$, we can deduce that  $V(c)$ is a vertex operator algebra. By Definition \ref{VAmod} and Theorem \ref{thm4.3}, we immediately obtain the following result.
	\begin{coro}\label{pro:weak-module}
		Let $W$ be a simple restricted module of $\Glm$ with central charge $c$. Then $W$ naturally carries the structure of a simple weak $V(c)$-module.
	\end{coro}
	
	For a vector space $W$, set
	$$
	W\{x\} =\Big\{\sum_{\alpha\in\C}w_{\alpha}x^{\alpha}\,|\,w_{\alpha}\in W\Big\}.
	$$
	\begin{defi}
		Let $V$ be a vertex algebra and let $\sigma$ be an automorphism of $V$ on which $\sigma$ acts semisimply. A \emph{$\sigma$-twisted $V$-module} is a vector space $W$ equipped with a linear map
		$$
		Y_W(\cdot,x)\ :\ V\rightarrow\big(\mathrm{End}(W)\big)\{x\}
		$$
		satisfying the conditions that $Y_W(\mathbf{1},x) =\mathrm{Id}_W$, $x^{\alpha} Y_W(u,x) \in \mathcal{E}(W)$ and
		\begin{align*}
			x_{0}^{-1} \delta&\left(  \frac{x_{1}-x_{2}}{x_{0}}\right) Y_W(u, x_{1}) Y_W(v, x_{2}) -x_{0}^{-1} \delta\left(\frac{-x_{2}+x_{1}}{x_{0}}\right) Y_W(v, x_{2}) Y_W(u, x_{1}) \\
			&\quad = x_{2}^{-1} \delta\left(\frac{x_{1}-x_{0}}{x_{2}}\right)\left(\frac{x_{1}-x_{0}}{x_{2}}\right)^{-\alpha}Y_W\big( Y_W(u, x_{0}) v, x_{2}\big),
		\end{align*}
		where $\alpha\in\C$, $v\in V$ and $u\in V^{\alpha}:=\{u \in V\, |\, \sigma(u) = e^{2\pi\sqrt{-1}\alpha}u\}.$
	\end{defi}
	
	If $\mu+\frac{\la}{2}\neq 0$, Li and Sun give a more general conclusion in \cite{LS}.
	
	\begin{theo}[\cite{LS}, Theorem 3.8]\label{thm4.6}
		For any $\lambda,\mu,c \in \C$, there is a one-to-one correspondence between restricted $\Glm$-module $W$ of level $c$ and  $\sigma_{\frac{\mu}{2}+\la}$-twisted module $(W,L_{-1},Y_W)$ of the vertex algebra $V(c)$ with
		\begin{equation*}
			Y_W(L_{-2} \mathbf{1}, x)=L(x),\quad Y_W(Y_{-\frac{3}{2}} \mathbf{1}, x)=x^{\frac{\mu}{2}+\la}\tilde{Y}(x),\quad Y_W(M_{-1} \mathbf{1}, x)=x^{\mu+2\la}M(x).
		\end{equation*}
	\end{theo}
	By Theorems \ref{th3} and   \ref{thm4.6},   we immediately have the following result.
	\begin{prop}
		Let $c\in\mathbb{C}$.  Then simple restricted  $\mathcal{G}_{\lambda,\mu}$-modules given in Theorem  \ref{th3}   are precisely  simple  $\sigma_{\frac{\mu}{2}+\la}$-twisted $V(c)$-modules
		for  $\mu+\frac{\la}{2}\neq 0$.
	\end{prop}
	
	\section{Some examples}\label{Sec5}
	In this section, we study two important classes of restricted modules over the deformed Schr\"{o}dinger-Virasoro algebra $\Glm$: highest weight modules and Whittaker modules. Due to the presence of deformation parameters $\lambda$ and $\mu$, the simplicity criteria involve additional constraints on these parameters.
	
	\subsection{Highest weight modules and Verma modules}\label{sec5.1}
	We do not need the assumption $\Re(\mu)\in[-1,0)$ in this subsection.

	For $\Glm$ with $(\la,\mu)\in\Gamma_1$, since $[L_0,M_n]=(n+2\mu)M_n$, we have the Cartan subalgebra $\mathfrak{h}_{\mu}=\mathrm{Span}_{\C}\{L_0,\,\delta_{0,\mu}M_0,\,C\}$ of $\Glm$.
	
	If $\mu=0$, we can get the triangular decomposition $\G_{\la,0} = \G_{\la,0}^{-} \oplus \mathfrak{h}_{0} \oplus \G_{\la,0}^{+}$, where
	\begin{equation*}
		\G_{\la,0}^{+} = \bigoplus_{r\in\frac{1}{2}\Z_+}\G_{\la,0}^{(r)}, \quad \G_{\la,0}^{-} = \bigoplus_{r\in\frac{1}{2}\Z_+}\G_{\la,0}^{(-r)}.
	\end{equation*}
	Let $\mathfrak{b}_{0}=\G_{\la,0}^{+} \oplus \mathfrak{h}_{0}$. Take the $\mathfrak{b}_{0}$-module $H(\Delta, \gamma, c)=\C v$ with the parameter $(\Delta, \gamma, c)\in\C^3$ such that:
	\begin{equation*}
		L_0 v = \Delta v, \quad M_0 v = \gamma v, \quad C v = c v, \quad \G_{\la,0}^{+} v = 0.
	\end{equation*}
	One can check that $H(\Delta, \gamma, c)$ is a simple highest weight  $\G_{\la,0}^{(0,0)}$-module corresponding to the case $t=d_1=d_2=0$ in Theorem \ref{th1}. Thus we can get the simple Verma module
	$$
	V_{\la,0}(\Delta, \gamma, c) = \mathrm{Ind}\big(H(\Delta, \gamma, c)\big)
	$$
	over $\G_{\la,0}$.
	
	If $\mu\neq 0$, the Cartan subalgebra $\mathfrak{h}_{\mu}$ turns into $\mathrm{Span}_{\C}\{L_0,\,C\}$. Let $\Glm = \Glm^{-} \oplus \mathfrak{h}_{\mu} \oplus \Glm^{+}$, where
	\begin{equation*}
		\Glm^{+} = \bigoplus_{r\in\frac{1}{2}\Z_+}\Glm^{(r)}, \quad \Glm^{-} = (\bigoplus_{r\in\frac{1}{2}\Z_+}\Glm^{(-r)})\oplus \C M_{0}.
	\end{equation*}
	In order to use Theorem \ref{th1} to construct the Verma module of $\Glm$, we consider a simple highest weight $\Glm^{(d_1,d_2)}$-module $U$. It follows from $\mathfrak{h}_{\mu}\subset\Glm^{(d_1,d_2)}$  and $d_1,d_2\in\N$ that  $d_1=d_2=0$, which means $\Glm^{(d_1,d_2)}=\Glm^{+}\oplus\mathfrak{h}_{\mu} \oplus\C M_0$. Besides, by $\Glm^{+}U=0$ and the action of $M_t$ on $U$ is injective for some $t\in \N$, one can deduce that $t=0$. However, the case is not similar to $\mu=0$. We can conclude that the following result holds.
	
	\begin{prop}
		If $\mu\neq 0$, there does not exist any simple highest weight $(\Glm^{+}\oplus\mathfrak{h}_{\mu} \oplus\C M_0)$-module satisfying the condition that  the action of $M_0$ on the module is injective.
	\end{prop}
	\begin{proof}
		We now proceed by contradiction. If there exists a simple highest weight $(\Glm^{+}\oplus\mathfrak{h}_{\mu} \oplus\C M_0)$-module $U(\Delta, c)$, with the highest weight vector $u$ such that
		\begin{equation*}
			L_0 u = \Delta u, \quad C u = c u, \quad \G_{\la,\mu}^{+} v = 0
		\end{equation*}
		for some $\Delta, c\in\C$ and the action of  $M_0$ on $U(\Delta, c)$ is injective, then we denote $u_i=M_0^{i} u,\, \forall\, i\in\N$.
		
		By (\ref{def1.1}) and $2\mu M_0 u_i =[L_0,M_0]u_i=(L_0M_0-M_0L_0)u_i$, we have
		\begin{equation*}
			L_0 u_{i+1}= L_0 M_0 u_i= 2\mu M_0 u_i+ M_0 L_0 u_i=2\mu  u_{i+1}+ M_0 L_0 u_i,
		\end{equation*}
		where $i\in\N$. By induction on $i\in\N$, it is clear that $L_0 u_{i}= (\Delta+2i\mu) u_{i}$ for all $i\in\N$. Take the subspace $U'=\oplus_{i=1}^{\infty}\C u_i$. One can directly check that $U'$ is a proper submodule of $U(\Delta, c)$, which is a contradiction.
	\end{proof}
	
	Hence, if $\mu\neq 0$, we can only construct a reducible highest weight $(\Glm^{+}\oplus\mathfrak{h}_{\mu} \oplus\C M_0)$-module $U(\Delta, c)=\oplus_{i=0}^{\infty}\C u_i$ with the parameter $(\Delta, c)\in\C^2$ such that:
	\begin{equation*}
		L_0 u_i = \Delta u_i, \quad M_0 u_i = (\Delta+2i\mu) u_i, \quad C u_i = c u_i, \quad \Glm^{+} u_i = 0, \quad \forall\,i\in\N.
	\end{equation*}
	By Remark \ref{rema3.3}, $U(\Delta, c)$ corresponds to the case $t=d_1=d_2=0$. Therefore, we can get a reducible Verma  $\Glm$-module
	$$
	V_{\la,\mu}(\Delta, c) = \mathrm{Ind}\big(U(\Delta, c)\big),
	$$
	which has the proper submodule $V_{\la,\mu}^{(n)}(\Delta, c)=\mathrm{Ind}(\oplus_{i=n}^{\infty}\C u_i)$ for any $n\in\Z_+$. Then the quotient module $V_{\la,\mu}(\Delta, c)/$ $V_{\la,\mu}^{(1)}(\Delta, c)$ is a simple highest weight module with the highest weight $(\Delta, c)$.
	
	\subsection{Whittaker modules}
	Now, we consider the Whittaker modules induced from a character $\psi$ of the positive part.
	
	Let $\psi: \Glm^{+} \to \C$ be a Lie algebra homomorphism. For $\Glm$ with $(\la,\mu)\in\Gamma_1$, the condition $\psi([\Glm^{+}, \Glm^{+}]) = 0$ implies
	\begin{eqnarray*}
		&&\psi(L_{n+2})=n^{-1}\psi([L_{1},L_{n+1}])=0, \\
		&&\psi(M_{n+1})=n^{-1}\psi([Y_{n+\frac{1}{2}},Y_{\frac{1}{2}}])=0,
	\end{eqnarray*}
	and
	\begin{equation}\label{equ5.1}
		\big(1-(\la+1)n+2\mu\big)\psi(Y_{n+\frac{1}{2}})=2\psi([L_{n},Y_{\frac{1}{2}}])=0,
	\end{equation}
	where $n \in \Z_+$. Since $(\la,\mu)\in\Gamma_1$, we have
	\begin{equation}\label{equ5.2}
		i+\lambda (i+t)-2 \mu\neq 0,\quad \forall\ i\in\Z_+\textrm{ and }t\in\N.
	\end{equation}
	Taking $(i,t)=(n-1,1)$ in (\ref{equ5.2}), by (\ref{equ5.1}), one can deduce that $\psi(Y_{n+\frac{1}{2}})=0$ for all $n\geq 2$. If $n=1$, (\ref{equ5.1}) turns into
	$$
	(2\mu-\la)\psi(Y_{\frac{3}{2}})=0.
	$$
	Let $W=\C w$ be a one-dimensional vector space with
	$$xw=\psi(x)w,\quad C w = c w,\quad \forall\,\, x\in\Glm^{+}.$$
	Then $W$ is a simple $(\Glm^{+}\oplus \C C)$-module if $\psi(M_1)\neq 0$. Now we set
	$$
	\psi(L_{1})=\Delta_1,\quad
	\psi(L_{2})=\Delta_2,\quad
	\psi(Y_{\frac{1}{2}})=\nu_1,\quad
	\psi(Y_{\frac{3}{2}})=\delta_{\la,2\mu}\nu_2,\quad
	\psi(M_{1})=\gamma\neq 0,
	$$
	where $(\Delta_1, \Delta_2, \nu_1, \nu_2, \gamma)\in \C^5$. Thus, the induced module
	$$
	W(\Delta_1, \Delta_2, \nu_1, \nu_2, \gamma, c)=\mathcal{U}(\Glm^{(1,1)})\otimes_{\mathcal{U}(\Glm^{+}\oplus \C C)}W
	$$
	is a $\Glm^{(1,1)}$-module. For convenience, we denote
	$W(\psi)=
	W(\Delta_1, \Delta_2, \nu_1, \nu_2, \gamma, c)$.
	One can find the element of $W(\psi)$ is a finite sum of the form
	\begin{equation}\label{equ5.3}
		a_{i_1,i_2,j,k}M_{0}^{i_1}M_{-1}^{i_2}Y_{-\frac{1}{2}}^{j}L_{0}^{k}w,\quad \textrm{ for }i_1,i_2,j,k\in\N\textrm{ and }a_{i_1,i_2,j,k}\in\C.
	\end{equation}
	For any $w'\in W(\psi)$ as in \eqref{equ5.3}, we denote by $\mathrm{supp}(w')$ the set of all $(i_1,i_2,j,k)\in\N^{4}$  such that $a_{i_1,i_2,j,k}\neq 0$. Then, we denote by $\mathrm{deg}(w')$ the maximal element in $\mathrm{supp}(w')$ under the following total order
	\begin{equation*}\aligned
		(i_4,i_3,i_2,i_1)&\succ (j_4,j_3,j_2,j_1)\  \\
		&\Leftrightarrow \ \mathrm{ there\ exists} \ r\in\Z_+ \ \mathrm{such \ that} \ (j_s=i_s,\ \forall\, 1\leq s<r) \ \mathrm{and} \ i_r>j_r,
		\endaligned
	\end{equation*}
	where $(i_4,i_3,i_2,i_1),\,(j_4,j_3,j_2,j_1)\in\N^{4}$. Now, taking any nonzero $w'\in W(\psi)$, suppose that $\mathrm{deg}(w')=(i_1,i_2,j,k)\in\N^{4}$. By $(\la,\mu)\in\Gamma_1$ and Lemma \ref{LYM}, the following results can be proved by the same method in Theorem \ref{th1}. Hence we omit the proof here.
	
	\textbf{Case 1:} If $k > 0$ and $\mu\neq-\frac{1}{2}$, then $\mathrm{deg}\big((M_{1} - \gamma)w'\big) = (i_1,i_2,j,k-1)$.
	
	\textbf{Case 2:} If $k = 0$ and $j > 0$, then $\mathrm{deg}\big((Y_\frac{3}{2} - \delta_{\la,2\mu}\nu_2)w'\big) = (i_1,i_2,j-1,0)$.
	
	\textbf{Case 3:} If $j= k = 0$ and $i_2 > 0$, then $\mathrm{deg}\big((L_{2} - \Delta_{2})w'\big) = (i_1,i_2-1,0,0)$.
	
	\textbf{Case 4:} If $i_2= j= k = 0$, $i_1 > 0$ and $\la\neq 2 \mu$, then $\mathrm{deg}\big((L_{1} - \Delta_{1})w'\big) = (i_1-1,0,0,0)$.
	
	In conclusion, if $\mu\neq-\frac{1}{2}$ and $\la\neq 2 \mu$, then $w\in \mathcal{U}(\Glm^{(1,1)})w'$, which means $\mathcal{U}(\Glm^{(1,1)})w'=W(\psi)$. Hence, $W(\psi)$ is a simple $\Glm^{(1,1)}$-module.
	It is clear that $W(\psi)$ satisfies the conditions of Theorem \ref{th1} with $t = d_1 = d_2 =1$. Thus, by Theorem \ref{th1}, we obtain the corresponding simple induced module $T(\psi)=\mathrm{Ind}\big(W(\psi)\big)$.
	
	\begin{remark}
		If $\Glm$ is the Schr\"{o}dinger-Virasoro algebra, i.e., $(\lambda, \mu)= (0, 0)$, we can deduce that $M_{0}$ is the center element of $\Glm$. Let $M_{0}w=\gamma'w$ for some $\gamma'\in \C$. From \textbf{Cases 1--3} and $M_{0}w=\gamma'w$, one can check $W(\psi)$ is still a simple $\Glm^{(1,1)}$-module and the  induced module $T(\psi)$ is simple, which corresponds to the module  constructed in \cite{CHS}.
	\end{remark}
	
	\subsection{The cases of $(\la,\mu)$}
	Finally, we consider some special $\la,\mu$.
	
	The deformed Schr\"{o}dinger-Virasoro algebra $\Glm$ provides a unified framework for several infinite-dimensional Lie algebras. By specializing the deformation parameters $(\lambda, \mu)$, we recover the following structures:
	
	\begin{enumerate}[(1)]
		\item \textbf{The Original Schr\"{o}dinger-Virasoro Algebra:}
		
		If $(\lambda, \mu)= (0, 0)$,
		the Lie brackets in $(\ref{def1.1})$ reduce to:
		\begin{equation*}
			[L_m, Y_{n}] = (n - \frac{1}{2}m - \frac{1}{2}) Y_{m+n}, \quad [L_m, M_n] = n M_{m+n}, \quad \forall \,\,  m,n\in\Z.
		\end{equation*}
		This is precisely the twisted Schr\"{o}dinger-Virasoro algebra introduced by Roger and Unterberger \cite{RU}. In this case, the Whittaker modules constructed in Section 5.2 coincide with the results in \cite{CHS}. The vertex algebra $V( c)$ becomes the standard Schr\"{o}dinger-Virasoro vertex operator algebra.
		
		\item \textbf{The Deformed $\mathfrak{bms}_3$ Algebra and Relation to $W(2,2)$ Algebra:}
		
		If $(\lambda, \mu)= (1, \frac{3}{2})$, the actions of $L_m$ on $Y_{n-\frac{1}{2}}$ and $M_n$ become:
		\begin{equation*}
			[L_m, Y_{n-\frac{1}{2}}] =(n-m+1)Y_{m+n-\frac{1}{2}},\quad [L_m, M_n] = (n - m +3) M_{m+n}, \quad \forall \,\,  m,n\in\Z.
		\end{equation*}
		By Lemma \ref{le2.4}, we have the isomorphism
		$$\varphi\ :\ L_n\mapsto L_n,\ Y_{n-\frac{3}{2}}  \mapsto J_{n},\ M_{n-3}  \mapsto I_{n},\ C\mapsto C,$$
		where $n\in\Z$. One can verify that $\G_{1, \frac{3}{2}}$ is  isomorphic to the deformed $\mathfrak{bms}_3$ algebra \cite{CCRS}.
		Here, $\{L_m, M_n\}$ generates a structure isomorphic to the $W(2,2)$ algebra (without the second central charge for $M$) \cite{ZD}. More generally, this implies that $\Glm$ contains the $W(a,b)$ algebra as an important subalgebra with $(a,b)=(2\mu,-\la)$. Our Theorem \ref{th1} and  Theorem \ref{th2} thus generalize the simplicity results for $W(a,b)$-modules.
	\end{enumerate}
	
	\vskip30pt
	
	\section*{Author Contributions}
	All authors contribute equally to this work.
	
	\section*{Data Availability Statement}
	This manuscript has no associated data.
	\section*{Conflicts of Interest} The authors declared that they have no conflict of interest to
	this work.
	\section*{Acknowledgements}
	The first author Haibo Chen thanks Professor  Haisheng  Li for teaching a short-term course on vertex operator algebra at Jimei University during 2025, and informs him of Reference \cite{LS}.
	

\vskip10pt
	
	\begin{thebibliography}{a}
		
		\bibitem{ALZ}
		D. Adamovi\'{c}, R. Lü, K. Zhao, Whittaker modules for the affine Lie algebra $A_1^{(1)}$, Adv. Math., {\bf 289}, 438-479(2016).
		
		\bibitem{CCRS}
		R. Caroca, P. Concha, E. Rodríguez, P. Salgado-Rebolledo, Generalizing the $\mathfrak{bms}_3$ and 2D-conformal algebras by expanding the Virasoro algebra, Eur. Phys. J. C, {\bf 78}, 262–276(2018).
		
		\bibitem{CG}
		H.J. Chen, X. Guo, New simple modules for the Heisenberg-Virasoro algebra, J. Algebra, {\bf 390}, 77-86(2013).
		
		\bibitem{CGLW}
		H.J. Chen, L. Ge, Z. Li, L. Wang, Classical Whittaker modules for the affine Kac-Moody algebras $A_N^{(1)}$, Adv. Math., {\bf 454}(109874), 60 pp(2024).
		
		\bibitem{CHS}
		H.B. Chen, Y. Hong, Y. Su, A family of new simple modules over the Schr\"{o}dinger-Virasoro algebra, J. Pure Appl. Algebra, {\bf 222}(4), 900-913(2018).
		
		\bibitem{Cqf}
		Q. Chen, Simple Restricted Modules for the Deformed $\mathfrak{bms}_3$ Algebra, Mathematics, {\bf 11}(4), 982(2023).
		
		\bibitem{CY}
		Q. Chen, Y. Yao, Simple restricted modules for the universal central extension of the planar Galilean conformal algebra, J. Algebra, {\bf 634}, 698-721(2023).
		
		\bibitem{CYZ}
		Y. Chen, Y. Yao, K. Zhao, Simple smooth modules over the Ramond algebra and applications to vertex operator superalgebras, Math. Z., {\bf 310}(2), No. 28, 17 pp(2025).
		
		\bibitem{FGXZ}
		V. Futorny, X. Guo, Y. Xue, K. Zhao, Smooth representations of affine Kac-Moody algebras, Adv. Math., {\bf 481}(110559), 34 pp(2025).
		
		\bibitem{GG}
		D. Gao, Y. Gao, Representations of the planar Galilean conformal algebra,
		Comm. Math. Phys., {\bf391} , 199-221(2022).
		
		\bibitem{Hm}
		M. Henkel, Schr\"{o}dinger invariance and strongly anisotropic critical systems, J. Stat. Phys., \textbf{75}, 1023--1029(1994).
		
		\bibitem{Kac}
		V. Kac, {\it Vertex Algebras for Beginners}, University Lecture Series, vol. 10, Amer. Math. Soc., 1997.
		
		\bibitem{LPX}
		D. Liu, Y. Pei, L. Xia, A category of restricted modules for the Ovisenko-Roger algebra, Algebr. Represent. Theory, {\bf 25}(3), 777-791(2022).
		
		\bibitem{LPX2}
		D. Liu, Y. Pei, L. Xia, Simple restricted modules for Neveu-Schwarz algebra, J. Algebra, {\bf 546}, 341-356(2020).
		
		\bibitem{Lhs}
		H. Li, Local systems of vertex operators, vertex superalgebras and modules, J. Pure Appl. Algebra, {\bf 109}(2), 143-195(1996).
		
		\bibitem{LS}
		H. Li, J. Sun,    On certain generalizations of the Schr\"{o}dinger-Virasoro algebra, {\it   J. Math. Phys.}  {\bf 56}(12),  121701(2015).
		
		\bibitem{LL}
		J. Lepowsky,  H. Li, Introduction to Vertex Operator Algebras and Their Representations, Progress in Math., {\bf 227},  (2004).
		
		\bibitem{LS2}
		J. Li, Y. Su,  Representations of the Schr\"odinger-Virasoro algebras, {\it J. Math. Phys.,} {\bf 49}(5),  053512(2008).
		
		\bibitem{MZ}
		V. Mazorchuk, K. Zhao, Simple Virasoro modules which are locally finite over a positive part, Selecta Math., {\bf 20}(3), 839-854(2014).
		
		\bibitem{MNTZ}
		Y. Ma, K. Nguyen, S. Tantubay, K. Zhao, Characterization of simple restricted modules, {\it   J. Algebra,}  {\bf 636},  1-19(2023).
		
		\bibitem{RU}
		C. Roger, J. Unterberger, The Schr\"{o}dinger-Virasoro Lie group and algebra: representation theory and cohomological study, {\it Ann. Henri Poincar\'{e}}, {\bf 7}, 1477-1529(2006).

	\bibitem{TYZ} H. Tan, Y. Yao,  K. Zhao,  Classification of simple smooth modules over the Heisenberg-Virasoro algebra, {\it Proc. Roy. Soc. Edinburgh Sect. A}, {\bf  155}(4), 1321-1365 (2025).
		
		\bibitem{TZ}
		S.  Tan, X. Zhang, Automorphisms and Verma modules for generalized Schr\"{o}dinger-Virasoro algebras, {\it J. Algebra,} {\bf 322}(12), 1379-1394(2009).
		
		\bibitem{Uj}
		J. Unterberger, On vertex algebra representations of the Schr\"{o}dinger-Virasoro Lie algebra, {\it Nucl. Phys. B}, {\bf 823}, 320--371(2009).
		
		\bibitem{ZD}
		W. Zhang, C. Dong, $W$-algebra $W(2,2)$ and the vertex operator algebra $L(\frac{1}{2},0) \otimes L(\frac{1}{2},0)$, {\it Commun. Math. Phys.}, {\bf 285}, 991--1004(2009).
		
		\bibitem{ZT}
		X. Zhang, S. Tan, Unitary representations for the Schr\"{o}dinger-Virasoro Lie algebra, {\it J. Algebra Appl.}, {\bf 12}, 1250132(2013).
		
		\bibitem{ZTL}
		X. Zhang, S. Tan, H. Lian,  Whittaker modules for the Schr\"{o}dinger-Witt algebra, {\it J. Math. Phys.,} {\bf 51}(8), 083524(2010).
		
	\end{thebibliography}
\end{document}